\documentclass[11pt,reqno]{amsart}

\newtheorem{proposition}{Proposition}[section]
\newtheorem{lemma}[proposition]{Lemma}
\newtheorem{corollary}[proposition]{Corollary}
\newtheorem{theorem}[proposition]{Theorem}

\theoremstyle{definition}
\newtheorem{definition}[proposition]{Definition}
\newtheorem{example}[proposition]{Example}
\newtheorem{examples}[proposition]{Examples}
\newtheorem{remark}[proposition]{Remark}
\newtheorem{remarks}[proposition]{Remarks}

\newcommand{\thlabel}[1]{\label{th:#1}}
\newcommand{\thref}[1]{Theorem~\ref{th:#1}}
\newcommand{\selabel}[1]{\label{se:#1}}
\newcommand{\seref}[1]{Section~\ref{se:#1}}
\newcommand{\lelabel}[1]{\label{le:#1}}
\newcommand{\leref}[1]{Lemma~\ref{le:#1}}
\newcommand{\prlabel}[1]{\label{pr:#1}}
\newcommand{\prref}[1]{Proposition~\ref{pr:#1}}
\newcommand{\colabel}[1]{\label{co:#1}}
\newcommand{\coref}[1]{Corollary~\ref{co:#1}}
\newcommand{\relabel}[1]{\label{re:#1}}
\newcommand{\reref}[1]{Remark~\ref{re:#1}}
\newcommand{\exlabel}[1]{\label{ex:#1}}
\newcommand{\exref}[1]{Example~\ref{ex:#1}}
\newcommand{\delabel}[1]{\label{de:#1}}
\newcommand{\deref}[1]{Definition~\ref{de:#1}}
\newcommand{\eqlabel}[1]{\label{eq:#1}}
\newcommand{\equref}[1]{(\ref{eq:#1})}

\def\ot{\otimes}

\def\CC{{\mathbb C}}

\newcommand{\Cc}{\mathcal{C}}

\newcommand{\Mm}{\mathcal{M}}

\def\*C{{}^*\hspace*{-1pt}{\Cc}}
\def\text#1{{\rm {\rm #1}}}

\input xy
\xyoption {all} \CompileMatrices

\usepackage{amssymb}
\usepackage{color,amssymb,graphicx,amscd,amsmath}
\usepackage[colorlinks,urlcolor=blue,linkcolor=blue,citecolor=blue]{hyperref}

\begin{document}
\title[Hochschild products and non-abelian cohomology for algebras]
{Hochschild products and global non-abelian cohomology for
algebras. Applications}


\author{A. L. Agore}
\address{Faculty of Engineering, Vrije Universiteit Brussel, Pleinlaan 2, B-1050 Brussels, Belgium \textbf{and} ''Simion Stoilow'' Institute of Mathematics of the Romanian Academy, P.O. Box 1-764, 014700 Bucharest, Romania} \email{ana.agore@vub.ac.be and
ana.agore@gmail.com}

\author{G. Militaru}
\address{Faculty of Mathematics and Computer Science, University of Bucharest, Str.
Academiei 14, RO-010014 Bucharest 1, Romania}
\email{gigel.militaru@fmi.unibuc.ro and gigel.militaru@gmail.com}

\thanks{A.L. Agore is Postdoctoral Fellow of the Fund for Scientific
Research-Flanders (Belgium) (F.W.O. Vlaanderen). This work was
supported by a grant of the Romanian National Authority for
Scientific Research, CNCS-UEFISCDI, grant no. 88/05.10.2011.}

\subjclass[2010]{16D70, 16S70, 16E40} \keywords{The extension
problem, non-abelian cohomology, classification of algebras}


\maketitle

\begin{abstract} Let $A$ be a unital associative algebra over a field $k$,
$E$ a vector space and $\pi : E \to A$ a surjective linear map
with $V = {\rm Ker} (\pi)$. All algebra structures on $E$ such
that $\pi : E \to A$ becomes an algebra map are described and
classified by an explicitly constructed global cohomological type
object ${\mathbb G} {\mathbb H}^{2} \, (A, \, V)$. Any such
algebra is isomorphic to a Hochschild product $A \star V$, an
algebra introduced as a generalization of a classical
construction. We prove that ${\mathbb G} {\mathbb H}^{2} \, (A, \,
V)$ is the coproduct of all non-abelian cohomologies ${\mathbb
H}^{2} \, \, (A, \, (V, \cdot))$. The key object ${\mathbb G}
{\mathbb H}^{2} \, (A, \, k)$ responsible for the classification
of all co-flag algebras is computed. All Hochschild products $A
\star k$ are also classified and the automorphism groups ${\rm
Aut}_{\rm Alg} (A \star k)$ are fully determined as subgroups of a
semidirect product $A^* \, \ltimes \bigl(k^* \times {\rm Aut}_{\rm
Alg} (A) \bigl)$ of groups. Several examples are given as well as
applications to the theory of supersolvable coalgebras or Poisson
algebras. In particular, for a given Poisson algebra $P$, all
Poisson algebras having a Poisson algebra surjection on $P$ with a
$1$-dimensional kernel are described and classified.
\end{abstract}

\section*{Introduction}
Introduced at the level of groups by H\"{o}lder \cite{holder}, the
extension problem is a famous and still open problem to which a
vast literature was devoted (see \cite{adem} and the references
therein). Fundamental results obtained for groups \cite{adem, EM,
sch} served as a model for studying the extension problem for
several other fields such as Lie/Leibniz algebras \cite{CE, LoP},
super Lie algebras \cite{AMR2}, associative algebras \cite{Ev,
Hoch2}, Hopf algebras \cite{AD}, Poisson algebras \cite{hue},
Lie-Rinehart algebras \cite{CML, hpr} etc. The extension problem
is one of the main tools for classifying 'finite objects' and has
been a source of inspiration for developing cohomology theories in
all fields mentioned above. We recall the extension problem using
the language of category theory. Let $\Cc$ be a category having a
zero (i.e. an initial and final) object $0$ and for which it is
possible to define an exact sequence. Given $A$, $B$ two
\emph{fixed} objects of $\Cc$, the extension problem consists of
the following question:

Describe and classify all \emph{extensions of $A$ by $B$}, i.e.
all triples $(E, \, i, \, \pi)$ consisting of an object $E$ of
$\Cc$ and two morphisms in $\Cc$ that fit into an exact sequence
of the form:
\begin{eqnarray*} \eqlabel{extencat}
\xymatrix{ 0 \ar[r] & B \ar[r]^{i} & E \ar[r]^{\pi} & A \ar[r] & 0
}
\end{eqnarray*}
Two extensions $(E, \, i, \, \pi)$ and $(E', \, i', \, \pi')$ of
$A$ by $B$ are called \emph{equivalent} if there exists an
isomorphism $\varphi: E \to E'$ in $\Cc$ that stabilizes $B$ and
co-stabilizes $A$, i.e. $\varphi \circ i = i'$ and $\pi' \circ
\varphi = \pi$. The answer to the extension problem is given by
explicitly computing the set ${\rm Ext} (A, \, B)$ of all
equivalence classes of extensions of $A$ by $B$ via this
equivalence relation. The simplest case is that of extensions with
an 'abelian' kernel $B$ for which a Schreier type theorem proves
that all extensions of $A$ by an abelian object $B$ are classified
by the second cohomology group ${\rm H}^2 (A, \, B)$ - the result
is valid for groups, Lie/Leibniz/associative/Poisson/Hopf
algebras, but the construction of the second cohomology group is
different for each of the above categories \cite{AD, CE, EM,
Hoch2, sch}. The difficult part of the extension problem is the
case when $B$ is not abelian: as a general principle, the Schreier
type theorems remain valid, but this time the classifying object
of all extensions of $A$ by $B$ is not the cohomology group of a
given complex anymore, but only a pointed set called the
\emph{non-abelian cohomology} ${\rm H}^2_{\rm nab} (A, \, B)$. For
its construction in the case of groups we refer to \cite{AB},
while for Lie algebras to \cite{fr} where it was proved that the
non-abelian cohomology ${\rm H}^2_{\rm nab} (A, \, B)$ is the
Deligne groupoid of a suitable differential graded Lie algebra.
The difficulty of the problem consists in explicitly computing
${\rm H}^2_{\rm nab} (A, \, B)$: lacking an efficient cohomology
tool as in the abelian case \cite{redondo, weibel}, it needs to be
computed 'case by case' using different computational and
combinatorial approaches.

This paper deals, at the level of associative algebras, with a
generalization of the extension problem, called the \emph{global
extension (GE) problem}, introduced recently in \cite{am-2015,
Mi2013} for Poisson/Leibniz algebras as a categorical dual of the
extending structures problem \cite{am-2015b, am-2013, am-2014}.
The GE-problem can be formulated for any category $\Cc$ using a
simple idea: in the classical extension problem we drop the
hypothesis '$B$ is a fixed object in $\Cc$' and replace it by a
weaker one, namely '$B$ has a fixed dimension'. For example, if
$\Cc$ is the category of unital associative algebras over a field
$k$, the GE-problem can be formulated as follows: for a given
algebra $A$, classify all associative algebras $E$ for which there
exists a surjective algebra map $E \to A \to 0$ whose kernel has a
given dimension $\mathfrak{c}$ as a vector space. Of course, any
such algebra has $A \times V$ as the underlying vector space,
where $V$ is a vector space such that ${\rm dim} (V) =
\mathfrak{c}$. Among several equivalent possibilities for writing
down the GE-problem for algebras, we prefer the following:

\emph{Let $A$ be a unital associative algebra, $E$ a vector space
and $\pi : E \to A$ a linear epimorphism of vector spaces.
Describe and classify the set of all unital associative algebra
structures that can be defined on $E$ such that $\pi : E \to A$
becomes a morphism of algebras.}

By classification of two algebra structures $\cdot_E$ and
$\cdot_E'$ on $E$ we mean the classification up to an isomorphism
of algebras $(E, \, \cdot_E) \cong (E, \, \cdot_E')$ that
stabilizes $V := {\rm Ker} (\pi)$ and co-stabilizes $A$: we shall
denote by ${\rm Gext} (A, \, E)$ the set of equivalence classes of
all algebra structures on $E$ such that $\pi : E \to A$ is an
algebra map. Let us explain now the significant differences
between the GE-problem and the classical extension problem for
associative algebras whose study was initiated in \cite{Ev,
Hoch2}. Let $(E, \, \cdot_E)$ be a unital algebra structure on $E$
such that $\pi: (E, \, \cdot_E) \to A$ is an algebra map. Then
$(E, \, \cdot_E)$ is an extension of the unital algebra $A$ by the
associative algebra $V = {\rm Ker} (\pi)$, which is a non-unital
subalgebra (in fact a two-sided ideal) of $(E, \, \cdot_E)$.
However, the multiplication on $V$ is not fixed from the input
data, as in the case of the classical extension problem: it
depends essentially on the algebra structures on $E$ which we are
looking for. Thus the classical extension problem is in some sense
the \emph{'local'} version of the GE-problem. The partial answer
and the best result obtained so far for the classical extension
problem was given in \cite[Theorem 6.2]{Hoch2}: all algebra
structures $\cdot_E$ on $E$ such that $V$ is a two-sided ideal of
null square (i.e. $x \cdot_E y = 0$, for all $x$, $y\in V$ - that
is $V$ is an \emph{'abelian'} algebra) are classified by the
second Hochschild cohomological group ${\rm H}^2 (A, \, V)$. The
result, as remarkable as it is, has its limitations: for example,
if $A := {\rm M}_n (k)$ is the algebra of $n\times n$-matrices
then there is no associative algebra $E$ of dimension $1 + n^2$
for which we have a surjective algebra map $A \to {\rm M}_n (k)$
with a null square kernel - however, there is a large family of
such algebras with a projection on ${\rm M}_n (k)$, one of them
being of course the direct product of algebras ${\rm M}_n (k)
\times k$ (more details are given in \exref{matrici}). This
example provides enough motivation for studying the GE-problem and
the fact that it covers the missing part in the classical
Hochschild approach of the extension problem.

The paper is organized as follows. \seref{glalg} gives the
theoretical answer to the GE-problem in three steps. First of all,
in \prref{hocprod} we introduce a new product $A \star V = A
\star_{(\triangleleft, \, \triangleright, \,  \vartheta, \,
\cdot)} V $, associated to an algebra $A$ and a vector space $V$
connected by two 'actions' of $A$ on $V$, a 'cocycle' and an
associative multiplication $\cdot$ on $V$ satisfying several
axioms. We call the algebra $A \star V$ the \emph{Hochschild
product} since, in the particular case when $\cdot$ is the trivial
multiplication on $V$ the above product reduces to the one
introduced by Hochschild in \cite{Hoch2}: in this case the axioms
involved in the construction of $A \star V$ come down to the fact
that $V$ is an $A$-bimodule and $\vartheta : A \times A \to V$ is
an usual normalized $2$-cocycle. On the other hand, if the cocycle
$\vartheta$ is the trivial map then the associated Hochschild
product $A \star V$ is just the \emph{semidirect product} $A \# \,
V$ of algebras and the corresponding axioms show that $(V, \,
\cdot)$ is an associative algebra in the monoidal category $(
{}_A\Mm_A, \, -\ot_A - , \, A )$ of $A$-bimodules. The canonical
surjection $A \star V \to A$ is an algebra map having the kernel
$V$. \prref{hocechiv} proves the converse: any algebra structure
$\cdot_E$ which can be defined on the vector space $E$ such that
$\pi : (E, \cdot_E) \to A$ is a morphism of algebras is isomorphic
to a Hochschild product $A \star V$. Based on these results and a
technical lemma, the theoretical answer to the GE-problem is given
in \thref{main1222}: the classifying set ${\rm Gext} (A, \, E)$ is
parameterized by an explicitly constructed global cohomological
object ${\mathbb G} {\mathbb H}^{2} \, (A, \, V)$ and the
bijection between the elements of ${\mathbb G} {\mathbb H}^{2} \,
(A, \, V)$ and ${\rm Gext} (A, \, E)$ is given. On the route,
\coref{structura} proves that any finite dimensional algebra is
isomorphic to an iteration of Hochschild products of the form
$\bigl( \cdots \bigl( (S \star V_1)\star V_2 \bigl) \star \cdots
\star V_t\bigl)$, where $S$ is a finite dimensional simple algebra
and $V_1, \cdots, V_t$ are finite dimensional vector spaces.
\coref{desccompcon} provides a decomposition of ${\mathbb G}
{\mathbb H}^{2} \, (A, \, V)$ as the coproduct of all non-abelian
cohomologies ${\mathbb H}^{2} \, \bigl(A, \, (V, \, \cdot_{V}
)\bigl)$, which are classifying objects for the extensions of $A$
by all associative algebra structures $\cdot$ on $V$ -- the second
Hochschild cohomological group ${\rm H}^{2} \, (A, \, V)$ is the
most elementary piece among all components of ${\mathbb G}
{\mathbb H}^{2} \, (A, \, V)$. Computing the classifying object
${\mathbb G} {\mathbb H}^{2} \, (A, \, V)$ is a highly nontrivial
task: if $A$ is finite dimensional and $V := k^n$ this object
parameterizes the equivalence classes of all unital associative
algebras of dimension $n + {\rm dim} (A)$ that admit an algebra
surjection on $A$. In \seref{coflg} we shall identify a way of
computing this object for a class of algebras called
\emph{co-flag} algebras over $A$, i.e. algebras $E$ that have a
finite chain of surjective morphisms of algebras $A_n : = E
\stackrel{\pi_{n}}{\longrightarrow} A_{n-1} \, \cdots \,
\stackrel{\pi_{2}} {\longrightarrow} A_1 \stackrel{\pi_{1}}
{\longrightarrow} A_{0} := A$, such that ${\rm dim} ( {\rm Ker}
(\pi_{i}) ) = 1$, for all $i = 1, \cdots, n$. All co-flag algebras
over $A$ can be completely described and classified by a recursive
reasoning whose key step is given in \prref{hoccal1} and
\coref{hoccal1ab} where ${\mathbb G} {\mathbb H}^{2} \, (A, \, k)$
is computed and all Hochschild products $A \star k$ are described
by generators and relations. The less restrictive classification
of these algebras, i.e. only up to an agebra isomorphism, is given
in \thref{clasres} where the second classifying object ${\mathbb
H} {\mathbb O} {\mathbb C} \, (A, \, k)$ is computed: it
parameterizes the isomorphism classes of all Hochschild products
$A \star k$, that is, it classifies up to an isomorphism, all
algebras $B$ which admit a surjective algebra map $B \to A$ with a
$1$-dimensional kernel. As a bonus of our approach, the
automorphism groups ${\rm Aut}_{\rm Alg} (A \star k)$ are fully
determined in \coref{izoaut} as subgroups of a semidirect product
$A^* \, \ltimes \bigl(k^* \times {\rm Aut}_{\rm Alg} (A) \bigl)$
of groups. At this point we should mention that determining the
structure of the automorphism group of a given algebra is an old
problem, intensively studied and very difficult, arising from
invariant theory (see \cite{bavula, cek} and the references
therein). Several examples where both classifying objects
${\mathbb G} {\mathbb H}^{2} \, (A, \, k)$ and ${\mathbb H}
{\mathbb O} {\mathbb C} \, (A, \, k)$ are explicitly computed for
different algebras $A$ are worked out in details. For instance, if
$A = k[C_n]$, is the group algebra of the cyclic group $C_n$ of
order $n$, then ${\mathbb G} {\mathbb H}^{2} \, (k[C_n], \, k)
\cong \bigl(U_{n}(k) \times U_{n}(k)\bigl) \, \sqcup \, k^* $,
where $U_{n}(k)$ is the group of $n$-th roots of unity in $k$. The
classification of these algebras, by describing the classifying
object  ${\mathbb H} {\mathbb O} {\mathbb C} \, (k[C_n], \, k)$,
is also indicated and reduces the question to a challenging number
theory problem which depends heavily on the arithmetics of $n$ and
the base field $k$ -- it is also related to two intensively
studied problems in the theory of group rings, namely the
description of all invertible elements and the automorphism group
of a group algebra \cite{jan, mili, oli}. An intriguing example is
$A := \mathcal{T}_{n}(k)$, the algebra of upper triangular
$(n\times n)$-matrices. The global cohomological object ${\mathbb
G} {\mathbb H}^{2} \, (\mathcal{T}_{n}(k), \, k)$ is computed in
\exref{Anamatrici} being described by a very interesting set of
matrices of trace $0$. In particular, $|{\mathbb H} {\mathbb O}
{\mathbb C} \, (\mathcal{T}_{2}(k), \, k)| = 8$, i.e. up to an
isomorphism of algebras there exist exactly eight $4$-dimensional
algebras that have an algebra projection on the Heisenberg algebra
$\mathcal{T}_{2}(k)$. Applications for coalgebras and Poisson
algebras are given in \seref{aplicatii} based on the same idea:
namely, that of rephrasing the concepts and results of this paper
for coalgebras (resp. Poisson manifolds) via two different
contravariant functors $(-)^* := {\rm Hom}_k \, (-, \, k)$ (resp.
$C^{\infty} (-)$) from the category of coalgebras (resp. Poisson
manifolds) to the category of algebras (resp. Poisson algebras).
Having the theory of supersolvable Lie algebras \cite{barnes} as a
source of inspiration  we introduce the concept of a
\emph{supersolvable coalgebra} in such a way that a coalgebra $C$
is supersolvable if and only if the convolution algebra $C^*$ is a
co-flag algebra. In particular, \coref{clasfinal} classifies all
$3$-dimensional supersolvable coalgebras over a field of
characteristic $\neq 2$. On the other hand, for a given Poisson
algebra $P$, \thref{pcodim1} classifies up to an isomorphism all
Poisson algebras $Q$ which admit a Poisson surjection $Q \to P \to
0$ with a $1$-dimensional kernel. The result is the algebraic
counterpart of the classification problem of all Poisson manifolds
containing a given Poisson manifold $M$ of codimension $1$. As an
example, we show that there exist exactly six families of
$4$-dimensional Poisson algebras with a Poisson algebra surjection
of the Heisenberg-Poisson algebra ${\mathcal H} (3, \, k)$. For
applications and further motivation for studying Poisson algebras
we refer to \cite{gra1993, gra2013, LPV}.

\section{The global extension problem} \selabel{glalg}
\subsection*{Notations and terminology}
For two sets $X$ and $Y$ we shall denote by $X \sqcup Y$ their
coproduct in the category of sets, i.e. $X \sqcup Y$ is the
disjoint union of $X$ and $Y$. Unless otherwise specified, all
vector spaces, linear or bilinear maps are over an arbitrary field
$k$. A map $f: V \to W$ between two vector spaces is called the
trivial map if $f (v) = 0$, for all $v\in V$. By an algebra $A$ we
always mean a unital associative algebra over $k$ whose unit will
be denoted by $1_A$. All algebra maps preserve units and any
left/right $A$-module is unital. For an algebra $A$, ${\rm
Aut}_{\rm Alg} (A)$ denotes the group of algebra automorphisms of
$A$, ${\rm Alg} \,(A, \, k)$ is the space of all algebra maps $A
\to k$ while ${}_A\Mm_A$ stands for the category of $A$-bimodules.
If $(V, \, \triangleright, \, \triangleleft) \in {}_A\Mm_A$ is an
$A$-bimodule, then the \emph{trivial extension} of $A$ by $V$ is
the algebra $A \times V$, with the multiplication defined for any
$a$, $b\in A$, $x$, $y\in V$ by:
\begin{equation} \eqlabel{trivialext}
(a, \, x) \cdot (b, \, y) := \bigl(ab, \,\, a \triangleright y  +
x\triangleleft b \bigl)
\end{equation}
Let $A$ be an algebra, $E$ a vector space, $\pi : E \to A$ a
linear epimorphism of vector spaces with $V: = {\rm Ker} (\pi)$
and denote by $i: V \to E$ the inclusion map. We say that a linear
map $\varphi: E \to E$ \emph{stabilizes} $V$ (resp.
\emph{co-stabilizes} $A$) if the left square (resp. the right
square) of the following diagram
\begin{eqnarray} \eqlabel{diagrama1}
\xymatrix {& V \ar[r]^{i} \ar[d]_{Id} & {E}
\ar[r]^{\pi} \ar[d]^{\varphi} & A \ar[d]^{Id}\\
& V \ar[r]^{i} & {E}\ar[r]^{\pi } & A}
\end{eqnarray}
is commutative. Two unital associative algebra structures $\cdot $
and $\cdot'$ on $E$ such that $\pi : E \to A$ is a morphism of
algebras are called \emph{cohomologous} and we denote this by $(E,
\cdot) \approx (E, \cdot')$, if there exists an algebra map
$\varphi: (E, \cdot) \to (E, \cdot')$ which stabilizes $V$ and
co-stabilizes $A$. We can easily prove that any such morphism is
bijective and thus, $\approx$ is an equivalence relation on the
set of all algebra structures on $E$ such that $\pi : E \to A$ is
an algebra map and we denote by ${\rm Gext} \, (A, \, E)$ the set
of all equivalence classes via the equivalence relation $\approx$.
${\rm Gext} \, (A, \, E)$ is the classifying object for the
GE-problem. In what follows we will prove that ${\rm Gext} \, (A,
\, E)$ is parameterized by a global cohomological object ${\mathbb
G} {\mathbb H}^{2} \, (A, \, V)$ which will be explicitly
constructed. To start with, we introduce the following:

\begin{definition} \delabel{hocdat}
Let $A$ be an algebra and $V$ a vector space. A \emph{Hochschild
data} of $A$ by $V$ is a system $\Theta(A, V) = (\triangleright,
\, \triangleleft, \, \vartheta, \, \cdot)$ consisting of four
bilinear maps
$$
\triangleright \, \, : A \times V \to V, \quad \triangleleft : V
\times A \to V, \quad \vartheta \,\, : A \times A \to V, \quad
\cdot \, : V \times V \to V
$$
\end{definition}

For a Hochschild data $ \Theta(A, V) = (\triangleright, \,
\triangleleft, \, \vartheta, \, \cdot)$ we denote by $A \star V =
A \star_{(\triangleleft, \, \triangleright, \, \vartheta, \,
\cdot)} V $ the vector space $A \times V$ with the multiplication
defined for any $a$, $b\in A$ and $x$, $y \in V$ by:
\begin{equation} \eqlabel{hoproduct2}
(a, x) \star (b, y) := (ab, \, \vartheta (a, b) + a \triangleright
y + x \triangleleft b + x \cdot y)
\end{equation}
$A \star V$ is called the \emph{Hochschild product} associated to
$\Theta(A, V)$ if it is an associative algebra with the
multiplication given by \equref{hoproduct2} and the unit $(1_A, \,
0_V)$. In this case $ \Theta(A, V) = (\triangleright, \,
\triangleleft, \, \vartheta, \, \cdot)$ is called a
\emph{Hochschild system} of $A$ by $V$. The multiplication defined
by \equref{hoproduct2} is more general than the one appearing in
the proof of \cite[Theorem 6.2]{Hoch2} -- the latter arises as a
special case of $A \star V$ for $\cdot : V \times V \to V$ the
trivial map, that is $x\cdot y = 0$, for all $x$, $y \in V$.
Moreover, the GE-problem is the dual, in the sense of category
theory, of the extending structures problem studied for algebras
in \cite{am-2014}: hence, from this viewpoint the Hochschild
product $A \star V$ can be seen as a categorical dual of the
\emph{unified product} $A \ltimes V$ introduced in \cite[Theorem
2.2]{am-2014}. The necessary and sufficient conditions for $A
\star V$ to be a Hochschild product are given in the following:

\begin{proposition}\prlabel{hocprod}
Let $A$ be an algebra, $V$ a vector space and $\Theta(A, V) =
(\triangleright, \, \triangleleft, \, \vartheta, \, \cdot)$ a
Hochschild data of $A$ by $V$. Then $A \star V$ is a Hochschild
product if and only if the following compatibility conditions hold
for any $a$, $b$, $c \in A$ and $x$, $y\in V$:
\begin{enumerate}
\item[(H0)] $\vartheta (a, \, 1_{A}) = \vartheta (1_{A}, \, a) =
0$, \, $x \triangleleft 1_{A} = x$, \, $1_{A} \triangleright x =
x$

\item[(H1)] $(x \cdot y) \lhd a = x \cdot (y \lhd a)$

\item[(H2)] $(x \lhd a) \cdot y = x \cdot (a \triangleright y)$

\item[(H3)] $a \triangleright (x \cdot y) = (a \triangleright x)
\cdot y$

\item[(H4)] $(a \triangleright x) \lhd b = a \triangleright (x
\lhd b)$

\item[(H5)] $\vartheta(a, \, bc) - \vartheta(ab, \, c) =
\vartheta(a, \, b) \lhd c - a \triangleright \vartheta(b, \,c)$

\item[(H6)] $(ab) \triangleright x = a \triangleright (b
\triangleright x) - \vartheta(a, \, b) \cdot x$

\item[(H7)] $x \lhd (ab) = (x \lhd a) \lhd b - x \cdot
\vartheta(a, \, b)$

\item[(H8)] The bilinear map $\cdot : V \times V \to V$ is
associative.
\end{enumerate}
\end{proposition}

Before going into the proof, we make some comments on the
relations (H0)-(H8). The first relation in (H0) together with (H5)
show that $\vartheta$ is a normalized Hochschild $2$-cocycle. (H6)
and (H7) are deformations of the usual left and respectively right
$A$-module conditions: together with (H4) and the last two
relations of (H0) they measure how far $(V, \, \triangleright, \,
\triangleleft )$ is from being an $A$-bimodule. Finally, axioms
(H1)-(H3) are compatibilities between the associative
multiplication $\cdot$ on $V$ and the 'actions' $(\triangleright,
\, \triangleleft )$ of $A$ on $V$ which are missing in the
classical theory \cite{Hoch2} since $\cdot$ is the trivial map and
thus they are automatically fulfilled. When $\vartheta$ is the
trivial map, axioms (H1)-(H3) together with (H6)-(H8) imply that
$(V, \cdot)$ is an associative algebra in the monoidal category
${}_A\Mm_A = ({}_A\Mm_A, \, -\ot_A - , \, A )$ of $A$-bimodules
(see \exref{abelian} below).

\begin{proof}
To start with, we can easily prove that $(1_A, \, 0_V)$ is the
unit for the multiplication defined by \equref{hoproduct2} if and
only if (H0) holds. The rest of the proof relies on a detailed
analysis of the associativity condition for the multiplication
given by \equref{hoproduct2}. Since in $A \star V$ we have $(a, x)
= (a, 0) + (0, x)$, it follows that the associativity condition
holds if and only if it holds for all generators of $A \star V$,
i.e. for the set $\{(a, \, 0) ~|~ a \in A\} \cup \{(0, \, x) ~|~ x
\in V\}$. To save space we will illustrate only a few cases, the
rest of the details being left to the reader. For instance, the
associativity condition for the multiplication given by
\equref{hoproduct2} holds in \{(0, x), \, (0, y), \, (a, 0)\} if
and only if (H1) holds. Similarly, the associativity condition
holds in \{(0, x), \,(a, 0), \, (0, y)\} if and only if (H2) holds
while, the associativity condition holds in \{(0, x), \, (0, y),
\, (0, z)\} if and only if $\cdot : V \times V \to V$ is
associative.
\end{proof}

From now on a Hochschild system of $A$ by $V$ will be viewed as a
system of bilinear maps $\Theta(A, V) = (\triangleright, \,
\triangleleft, \, \vartheta, \, \cdot)$ satisfying the axioms
(H0)-(H8) and we denote by ${\mathcal H}{\mathcal S} \, (A, \, V)$
the set consisting of all Hochschild systems of $A$ by $V$.

\begin{examples} \exlabel{abelian}
1. By applying \prref{hocprod} we obtain that a Hochschild data
$(\triangleright, \, \triangleleft, \, \vartheta, \, \cdot)$ for
which $\cdot$ is the trivial map is a Hochschild system if and
only if $(V, \triangleright, \triangleleft)$ is an $A$-bimodule
and $\vartheta : A \times A \to V$ is a normalized $2$-cocycle.
Furthermore, if $\vartheta$ is also the trivial map, then the
associated Hochschild product is just the trivial extension of the
algebra $A$ by the $A$-bimodule $V$ as defined by
\equref{trivialext}.

2. A Hochschild system $\Theta(A, V) = (\triangleright, \,
\triangleleft, \, \vartheta, \, \cdot)$ for which $\vartheta$ is
the trivial map is called a \emph{semidirect system} of $A$ by
$V$. In this case $\vartheta$ will be omitted when writing down
$\Theta(A, V)$ and axioms in \prref{hocprod} take a simplified
form: $\Theta(A, V) = (\triangleright, \, \triangleleft, \,
\cdot)$ is a semidirect system if and only if $(V, \,
\triangleright, \, \triangleleft) \in {}_A\Mm_A$ is an
$A$-bimodule, $(V, \cdot)$ is an associative algebra and
\begin{eqnarray}
(x \cdot y) \lhd a = x \cdot (y \lhd a), \qquad a \triangleright
(x \cdot y) = (a \triangleright x) \cdot y, \qquad (x \lhd a)
\cdot y = x \cdot (a \triangleright y) \eqlabel{monoidal}
\end{eqnarray}
for all $a\in A$, $x$, $y\in V$. The Hochschild product associated
to a semidirect system $\Theta(A, V) = (\triangleright, \,
\triangleleft, \, \cdot)$ is called a \emph{semidirect product} of
algebras and will be denoted by $A \# V := A \#_{(\triangleleft,
\, \triangleright, \, \cdot)} V$. The terminology will be
motivated below in \coref{splialg}: exactly as in the case of
groups or Lie algebras, the semidirect product of algebras
describes split epimorphisms in the category of algebras. We will
rephrase the axioms of a semidirect system $\Theta(A, V) =
(\triangleright, \, \triangleleft, \, \cdot)$ of $A$ by $V$ using
the language of monoidal categories. The first and the second
axioms of \equref{monoidal} are equivalent to the fact that the
bilinear map $\cdot : V \times V \to V$ is an $A$-bimodule map,
while the last one is the same as saying that the map is
$A$-balanced. The space of these maps is in one-to-one
correspondence with the set of all $A$-bimodule maps $V\ot_A V \to
V$. This fact together with the other two axioms can be rephrased
as follows: $\Theta(A, V) = (\triangleright, \, \triangleleft, \,
\cdot)$ is a semidirect system of $A$ by $V$ if and only if $(V,
\cdot)$ is a (not-necessarily unital) associative algebra in the
monoidal category ${}_A\Mm_A = ({}_A\Mm_A, \, -\ot_A - , \, A )$
of $A$-bimodules.
\end{examples}

The Hochschild product is the tool to answer the GE-problem.
Indeed, first we observe that the canonical projection $\pi_A : A
\star V \to A$, $\pi_A (a, x) := a$ is a surjective algebra map
with kernel $\{0\} \times V \cong V$. Hence, the algebra $A \star
V$ is an extension of the algebra $A$ by the associative algebra
$(V, \cdot)$ via
\begin{eqnarray} \eqlabel{extenho1}
\xymatrix{ 0 \ar[r] & V \ar[r]^{i_{V}} & A \star \, V
\ar[r]^{\pi_{A}} & A \ar[r] & 0 }
\end{eqnarray}
where $i_V (x) = (0, x)$. Conversely, we have:

\begin{proposition}\prlabel{hocechiv}
Let $A$ be an algebra, $E$ a vector space and $\pi : E \to A$ an
epimorphism of vector spaces with $V = {\rm Ker} (\pi)$. Then any
algebra structure $\cdot$ which can be defined on the vector space
$E$ such that $\pi : (E, \cdot) \to A$ becomes a morphism of
algebras is isomorphic to a Hochschild product $A \star V$ and
moreover, the isomorphism of algebras $ (E, \cdot) \cong A \star
V$ can be chosen such that it stabilizes $V$ and co-stabilizes
$A$.

Thus, any unital associative algebra structure on $E$ such that
$\pi : E \to A$ is an algebra map is cohomologous to an extension
of the form \equref{extenho1}.
\end{proposition}

\begin{proof}
Let $\cdot$ be an algebra structure of $E$ such that $\pi: (E,
\cdot) \to A$ is an algebra map. Since $k$ is a field we can pick
a $k$-linear section $s : A \to E$ of $\pi$, i.e. $\pi \circ s =
{\rm Id}_{A}$ and $s (1_A) = 1_E$. Then $\varphi : A \times V \to
E$, $\varphi (a, x) := s(a) + x$ is an isomorphism of vector
spaces with the inverse $\varphi^{-1} (y) = \bigl(\pi(y), \, y - s
(\pi(y)) \bigl)$, for all $y\in E$. Using the section $s$ we
define three bilinear maps given for any $a$, $b\in A$ and $x\in
V$ by:
\begin{eqnarray*}
\triangleleft &=& \triangleleft_{s} \,\,\, : V \times A \to V,
\,\,\,\,\,\,\, x
\triangleleft a := x \cdot s(a) \eqlabel{act1}\\
\triangleright &=& \triangleright_{s} \, \, : A \times V \to V,
\,\,\,\,
a \triangleright x := s(a) \cdot x \eqlabel{act2}\\
\vartheta &=& \vartheta_s \,\,\, : A \times A \to V, \,\,\,\,
\vartheta (a, b) := s(a) \cdot s(b) - s(ab)  \eqlabel{coc1}
\end{eqnarray*}
and let $\cdot_V : V \times V \to V$ be the restriction of $\cdot$
at $V$, i.e. $x \cdot_V \, y := x \cdot y$, for all $x$, $y\in V$.
We can easily see that they are well-defined maps. The key step is
the following: using the system $(\triangleleft, \,
\triangleright, \, \vartheta, \, \cdot_V = \cdot)$ connecting $A$
and $V$ we can prove that the unique algebra structure $\star $
that can be defined on the direct product of vector spaces $A
\times V$ such that $\varphi : A \times V \to (E, \cdot)$ is an
isomorphism of algebras is given by:
\begin{equation} \eqlabel{hoproduct}
(a, x) \star (b, y) := (ab, \, \vartheta (a, b) + a \triangleright
y + x \triangleleft b + x \cdot y)
\end{equation}
for all $a$, $b\in A$, $x$, $y \in V$. Indeed, let $\star $ be
such an algebra structure on $A\times V$. Then:
\begin{eqnarray*}
(a, x) \star (b, y) &=& \varphi^{-1} \bigl(\varphi(a, x) \cdot
\varphi(b, y)\bigl) = \varphi^{-1} \bigl(  ( s(a) + x) \cdot (s(b) + y) \bigl) \\
&=& \varphi^{-1} ( s(a) \cdot s(b) + s(a) \cdot y + x \cdot s(b) + x \cdot y ) \\
&=& \bigl ( ab, \,  s(a) \cdot s(b) - s (ab) + s(a) \cdot y + x \cdot s(b) + x\cdot y \bigl) \\
&=& \bigl( ab, \, \vartheta (a, b) + a \triangleright y + x
\triangleleft b + x \cdot y )
\end{eqnarray*}
as needed. Thus, $\varphi : A \star V \to (E, \cdot)$ is an
isomorphism of algebras and we can see that it stabilizes $V$ and
co-stabilizes $A$.
\end{proof}

Using \prref{hocechiv} we obtain the following result concerning
the structure of finite dimensional algebras which indicates the
crucial role played by Hochschild products. We can survey all
algebras of a given dimension if we are able to compute various
Hochschild systems starting with a simple algebra (whose structure
is well-known due to the Wederburn-Artin theorem) and the
associated Hochschild products. It is the associative algebra
counterpart of a similar result from group theory \cite[pages
283-284]{rotman}.

\begin{corollary}\colabel{structura}
Any finite dimensional algebra is isomorphic to an iteration of
Hochschild products of the form $\Bigl( \cdots \bigl( (S \star
V_1)\star V_2 \bigl) \star \cdots \star V_t\Bigl)$, where $S$ is a
finite dimensional simple $k$-algebra, $t$ is a positive integer
and $V_1, \cdots, V_t$ are finite dimensional vector spaces.
\end{corollary}

\begin{proof}
Let $A$ be an algebra of dimension $n$. The proof goes by
induction on $n$. If $n = 1$ then $A \cong k \cong k \star \{0\}$
and $k$ is a simple algebra. Assume now that $n > 1$. If $A$ is
simple there is nothing to prove. On the contrary, if $A$ has a
proper two-sided ideal $\{0\} \neq V_t \neq A$, let $\pi : A \to
A_1 := A/V_t$ be the canonical projection. It follows from
\prref{hocechiv} that $A \cong A_1 \star V_t$, for some Hochschild
system of $A_1$ by $V_t$. If $A_1$ is simple the proof is
finished; if $A_1$ is not simple, we apply induction since ${\rm
dim}_k (A_1) < n$.
\end{proof}

The semidirect products of algebras characterize split epimorphism
in this category:

\begin{corollary} \colabel{splialg}
An algebra map $\pi: B \to A$ is a split epimorphism in the
category of algebras if and only if there exists an isomorphism of
algebras $B \cong A \# V$, where $V = {\rm Ker} (\pi)$ and $A \#
V$ is a semidirect product of algebras.
\end{corollary}

\begin{proof}
First we note that for a semidirect product $A\# V$, the canonical
projection $p_A : A \# V \to A$, $p_{A}(a, \, x) = a$ has a
section that is an algebra map defined by $s_A (a) = (a, 0)$, for
all $a\in A$. Conversely, let $s: A \to B$ be an algebra map such
that $\pi \circ s = {\rm Id}_A$. Then, the bilinear map
$\vartheta_s$ constructed in the proof of \prref{hocechiv} is the
trivial map and hence the corresponding Hochschild product $A
\star V$ is a semidirect product $A \# V$.
\end{proof}

\prref{hocechiv} shows that the classification part of the
GE-problem reduces to the classification of all Hochschild
products associated to all Hochschild systems of $A$ by $V$. This
is what we will do next.

\begin{lemma}\lelabel{HHH}
Let $\Theta(A, V) = (\triangleright, \, \triangleleft, \,
\vartheta, \, \cdot)$ and $\Theta '(A, V) = (\triangleright ', \,
\triangleleft ', \, \vartheta ', \, \cdot ')$ be two Hochschild
systems and $A \star V$, respectively $A \star ' V$, the
corresponding Hochschild products. Then there exists a bijection
between the set of all morphisms of algebras $\psi: A \star V \to
A \star ' V$ which stabilize $V$ and co-stabilize $A$ and the set
of all linear maps $r: A \to V$ with $r(1_{A}) = 0$ satisfying the
following compatibilities for all $a$, $b \in A$, $x$, $y \in V$:
\begin{enumerate}
\item[(CH1)] $x \cdot y = x \cdot ' y$; \item[(CH2)] $x
\triangleleft a = x \triangleleft ' a + x \cdot ' r(a)$;
\item[(CH3)] $a \triangleright x = a \triangleright ' x + r(a)
\cdot ' x$; \item[(CH4)] $\vartheta(a, \,b) + r(ab) = \vartheta
'(a, \, b) + a \triangleright ' r(b) + r(a) \lhd ' b + r(a) \cdot
' r(b)$
\end{enumerate}
Under the above bijection the morphism of algebras $\psi =
\psi_{r}: A \star V \to A \star ' V$ corresponding to $r: A \to V$
is given by $\psi(a, x) = (a, \, r(a) + x)$, for all $a \in A$, $x
\in V$. Moreover, $\psi = \psi_{r}$ is an isomorphism with the
inverse given by $\psi^{-1}_{r} = \psi_{-r}$.
\end{lemma}

\begin{proof} It is an elementary fact that a linear map $\psi: A \times V \to A \times V$
stabilizes $V$ and co-stabilizes $A$ if and only if there exists a
uniquely determined linear map $r: A \to V$ such that $\psi(a, x)
= (a, r(a)+ x)$, for all $a \in A$, $x \in V$. Let $\psi =
\psi_{r}$ be such a linear map. We will prove that $\psi: A \star
V \to A \star ' V$ is an algebra map if and only if $r(1_{A}) = 0$
and the compatibility conditions (CH1)-(CH4) hold. To start with
it is straightforward to see that $\psi$ preserve the unit $(1_A,
\, 0)$ if and only if $r(1_{A}) = 0$. The proof will be finished
if we check that the following compatibility holds for all
generators of $A \times V$:
\begin{equation}\eqlabel{algebramap222}
\psi\bigl((a, x) \star (b, y)\bigl) = \psi\bigl((a, x)\bigl) \,
\star ' \, \psi\bigl((b, y)\bigl)
\end{equation}
By a straightforward computation it follows that
\equref{algebramap222} holds for the pair $(a, 0)$, $(b, 0)$ if
and only if (CH4) is fulfilled while \equref{algebramap222} holds
for the pair $(0, x)$, $(a, 0)$ if and only if (CH2) is satisfied.
Finally, \equref{algebramap222} holds for the pair $(a, 0)$, $(0,
x)$ and respectively $(0, x)$, $(0, y)$ if and only if (CH3) and
respectively (CH1) hold.
\end{proof}

\leref{HHH} leads to the following:

\begin{definition}\delabel{echiaa}
Let $A$ be an algebra and $V$ a vector space. Two Hochschild
systems $\Theta(A, V)=(\triangleright, \, \triangleleft, \,
\vartheta, \, \cdot)$ and $\Theta'(A, V) = (\triangleright ', \,
\triangleleft ', \, \vartheta ', \, \cdot ')$ are called
\emph{cohomologous}, and we denote this by $\Theta(A, V) \approx
\Theta'(A, V)$, if and only if $\cdot = \cdot '$ and there exists
a linear map $r: A \to V$ such that $r(1_{A}) = 0$ and for any
$a$, $b \in A$, $x$, $y \in V$ we have:
\begin{eqnarray}
x \triangleleft a &=& x \triangleleft ' a + x \cdot ' r(a) \eqlabel{compa1}\\
a \triangleright x &=& a \triangleright ' x + r(a) \cdot ' x \eqlabel{compa2}\\
\vartheta(a, \,b) &=& \vartheta '(a, \, b) - r(ab) + a
\triangleright ' r(b) + r(a) \lhd ' b + r(a) \cdot ' r(b)
\eqlabel{compa3}
\end{eqnarray}
\end{definition}

As a conclusion of this section, we obtain the theoretical answer
to the GE-problem:

\begin{theorem}\thlabel{main1222}
Let $A$ be an algebra, $E$ a vector space and $\pi : E \to A$ a
linear epimorphism of vector spaces with $V = {\rm Ker} (\pi)$.
Then $\approx$ defined in \deref{echiaa} is an equivalence
relation on the set ${\mathcal H} {\mathcal S} (A, V)$ of all
Hochschild systems of $A$ by $V$. If we denote by ${\mathbb G}
{\mathbb H}^{2} \, (A, \, V) := {\mathcal H}\mathcal{S} (A, V)/
\approx $, then the map
\begin{equation} \eqlabel{thprinc}
{\mathbb G} {\mathbb H}^{2} \, (A, \, V) \to {\rm Gext} \, (A, \,
E), \qquad \overline{(\triangleright, \, \triangleleft, \,
\vartheta, \, \cdot)} \, \longmapsto \, A \star_{(\triangleleft,
\, \triangleright, \, \vartheta, \, \cdot)} \, V
\end{equation}
is bijective, where $\overline{(\triangleright, \, \triangleleft,
\, \vartheta, \, \cdot)}$ denotes the equivalence class of
$(\triangleright, \, \triangleleft, \, \vartheta, \, \cdot)$ via
$\approx$. 
\end{theorem}

\begin{proof} Follows from \prref{hocprod}, \prref{hocechiv} and \leref{HHH}.
\end{proof}

Computing the classifying object ${\mathbb G} {\mathbb H}^{2} \,
(A, \, V)$, for a given algebra $A$ and a given vector space $V$
is a very difficult task. The first step in decomposing this
object is suggested by the way the equivalence relation $\approx$
was introduced in \deref{echiaa}: it shows that two different
associative algebra structures $\cdot$ and $\cdot^{'}$ on $V$ give
two different equivalence classes of the relation $\approx$ on
${\mathcal H} {\mathcal S} (A, V)$. Let us fix $\cdot_{V}$ an
associative multiplication on $V$ and denote by ${\mathcal H}
{\mathcal S}_{\cdot_{V}} (A, \, V)$ the set of all triples
$\bigl(\triangleleft, \, \triangleright, \, \vartheta \bigl)$ such
that $\bigl( \triangleright, \, \triangleleft, \, \vartheta, \,
\cdot_V \bigl) \in {\mathcal H} {\mathcal S} (A, V)$. Two triples
$\bigl(\triangleleft, \, \triangleright, \, \vartheta \bigl)$ and
$\bigl(\triangleleft', \, \triangleright', \, \vartheta' \bigl)
\in {\mathcal H} {\mathcal S}_{\cdot_{V}} (A, \, V)$ are
$\cdot_V$-cohomologous and we denote this by $\bigl(\triangleleft,
\, \triangleright, \, \vartheta \bigl) \, \approx_{\cdot_V}
\bigl(\triangleleft', \, \triangleright', \, \vartheta' \bigl)$ if
there exists a linear map $r: A \to V$ such that $r(1_{A}) = 0$
and the compatibility conditions \equref{compa1}-\equref{compa3}
are fulfilled for $\cdot' = \cdot_V$. Then $\approx_{\cdot_V}$ is
an equivalence relation on ${\mathcal H} {\mathcal S}_{\cdot_{V}}
(A, \, V)$ and we denote by ${\mathbb H}^{2} \, \bigl(A, \, (V, \,
\cdot_{V} )\bigl)$ the quotient set ${\mathcal H} {\mathcal
S}_{\cdot_{V}} (A, \, V)/ \approx_{\cdot_V}$. The non-abelian
cohomology ${\mathbb H}^{2} \, \bigl(A, \, (V, \, \cdot_{V}
)\bigl)$ classifies all extensions of the unital associative
algebra $A$ by a fixed associative algebra $(V, \, \cdot_{V})$.
This can be seen as a Schreier type theorem for associative
algebras: we mention that \cite[Theorem 6.2]{Hoch2} (see also
\cite[Proposition 3.7]{redondo}) follows as a special case of
\coref{schclas} if we let $\cdot_V$ to be the trivial map.

\begin{corollary} \colabel{schclas}
Let $A$ be a unital associative algebra and $(V, \cdot_V)$ an
associative multiplication on $V$. Then, the map
\begin{equation} \eqlabel{clasextprob}
{\mathbb H}^{2} \, \bigl(A, \, (V, \, \cdot_{V} )\bigl) \to {\rm
Ext} \, (A, \, (V, \, \cdot_V) ), \qquad
\overline{\overline{(\triangleright, \, \triangleleft, \,
\vartheta)}} \, \longmapsto \, A \star_{(\triangleleft, \,
\triangleright, \, \vartheta, \, \cdot_V)} \, V
\end{equation}
is bijective, where ${\rm Ext} \, (A, \, (V, \, \cdot_V) )$ is the
set of equivalence classes of all unital associative algebras that
are extensions of the algebra $A$ by $(V, \cdot_V)$ and
$\overline{\overline{(\triangleright, \, \triangleleft, \,
\vartheta)}}$ denotes the equivalence class of $(\triangleright,
\, \triangleleft, \, \vartheta)$ via $\approx_l$.
\end{corollary}

The above considerations give also the following decomposition of
${\mathbb G} {\mathbb H}^{2} \, (A, \, V)$:

\begin{corollary} \colabel{desccompcon}
Let $A$ be an algebra, $E$ a vector space and $\pi : E \to A$ an
epimorphism of vector spaces with $V = {\rm Ker} (\pi)$. Then
\begin{equation}\eqlabel{balsoi}
{\mathbb G} {\mathbb H}^{2} \, (A, \, V) = \, \sqcup_{\cdot_{V}}
\, {\mathbb H}^{2} \, \bigl(A, \, (V, \, \cdot_{V} )\bigl)
\end{equation}
where the coproduct on the right hand side is in the category of
sets over all possible associative algebra structures $\cdot_{V}$
on the vector space $V$.
\end{corollary}

By looking at formula \equref{balsoi} one can see that computing
${\mathbb G} {\mathbb H}^{2} \, (A, \, V)$ is a very laborious
task in which the first major barrier is describing all
associative multiplications on $V$. The complexity of the
computations involved increases along side with ${\rm dim} (V)$.
Among all components of the coproduct in \equref{balsoi} the
simplest one is that corresponding to the trivial associative
algebra structure on $V$, i.e. $x \cdot_V y := 0$, for all $x$, $y
\in V$. We shall denote this trivial algebra structure on $V$ by
$V_0 := (V, \, \cdot_V = 0)$ and we shall prove that ${\mathbb
H}^{2} \, \bigl(A, \, V_0 \bigl)$ is the coproduct of all
classical second cohomological groups. Indeed, let ${\mathcal H}
{\mathcal S}_0 (A, \, V_0)$ be the set of all triples
$\bigl(\triangleleft, \, \triangleright, \, \vartheta \bigl)$ such
that $\bigl( \triangleright, \, \triangleleft, \, \vartheta, \,
\cdot_V := 0 \bigl) \in {\mathcal H} {\mathcal S} (A, V)$.
\exref{abelian} shows that a triple $\bigl(\triangleleft, \,
\triangleright, \, \vartheta \bigl) \in {\mathcal H} {\mathcal
S}_0 (A, \, V_0)$ if and only if $(V, \triangleright,
\triangleleft)$ is an $A$-bimodule and $\vartheta : A \times A \to
V$ is a normalized $2$-cocycle. Two triples $\bigl(\triangleleft,
\, \triangleright, \, \vartheta \bigl)$ and $\bigl(\triangleleft',
\, \triangleright', \, \vartheta' \bigl) \in {\mathcal H}
{\mathcal S}_0 (A, \, V_0)$ are $0$-cohomologous
$\bigl(\triangleleft, \, \triangleright, \, \vartheta \bigl)
\approx_0 \bigl(\triangleleft', \, \triangleright', \, \vartheta'
\bigl)$ if and only if $\triangleleft = \triangleleft'$,
$\triangleright = \triangleright'$ and there exists a linear map
$r: A \to V$ such that $r(1_{A}) = 0$ and
\begin{equation}\eqlabel{cazslab}
\vartheta(a, \, b) = \vartheta '(a, \, b) - r(ab) + a
\triangleright  r(b) + r(a) \triangleleft b
\end{equation}
for all $a$, $b\in A$ -- these are the conditions remaining from
\deref{echiaa} applied for the trivial multiplication $\cdot :=
0$. The equalities $\triangleleft = \triangleleft'$ and
$\triangleright = \triangleright'$ show that two different
$A$-bimodule structures over $V$ give different equivalence
classes in the classifying object ${\mathbb H}^{2} \, \bigl(A, \,
V_0 \bigl)$. Thus, for computing it we can also fix $(V, \,
\triangleleft, \, \triangleright)$ an $A$-bimodule structure over
$V$ and consider the set ${\rm Z}^2_{(\triangleleft, \,
\triangleright)} \, (A, \, V_0) $ of all normalized Hochschild
$2$-cocycles: i.e. bilinear maps $\vartheta : A \times A \to V$
satisfying (H5) and the first condition of (H0). Two normalized
$2$-cocycles $\vartheta$ and $\vartheta'$ are cohomologous
$\vartheta \approx_0 \vartheta'$ if and only if there exists a
linear map $r: A \to V$ such that $r(1_{A}) = 0$ and
\equref{cazslab} holds. $\approx_0$ is an equivalence relation on
the set ${\rm Z}^2_{(\triangleleft, \, \triangleright)} \, (A, \,
V_0) $ and the quotient set ${\rm Z}^2_{(\triangleleft, \,
\triangleright)} \, (A, \, V_0)/ \approx_0$ is just the classical
second Hochschild cohomological group which we denote by ${\rm
H}^2_{(\triangleleft, \, \triangleright)} \, (A, \, V_0)$. All the
above considerations prove the following:

\begin{corollary}\colabel{cazuabspargere}
Let $A$ be an algebra and $V$ a vector space viewed with the
trivial associative algebra structure $V_0$. Then:
\begin{equation}\eqlabel{balsoi2}
{\mathbb H}^{2} \, \bigl(A, \, V_0 \bigl) \, = \,
\sqcup_{(\triangleleft, \, \triangleright)} \, {\rm
H}^2_{(\triangleleft, \, \triangleright)} \, (A, \, V_0)
\end{equation}
where the coproduct on the right hand side is in the category of
sets over all possible $A$-bimodule structures $(\triangleleft, \,
\triangleright)$ on the vector space $V$.
\end{corollary}

\section{Co-flag algebras. Examples.} \selabel{coflg}

In this section we apply the theoretical results obtained in
\seref{glalg} for some concrete examples: more precisely, for a
given algebra $A$ we shall classify all unital associative
algebras $B$ such that there exists a surjective algebra map $\pi
: B \to A$ having a $1$-dimensional kernel, which as a vector
space will be assumed to be $k$. First, we shall compute ${\mathbb
G} {\mathbb H}^{2} \, (A, \, k)$: it will classify all these
algebras up to an isomorphism which stabilizes $k$ and
co-stabilizes $A$. Then, we will compute the second classifying
object, denoted by ${\mathbb H} {\mathbb O} {\mathbb C} \, (A, \,
k)$, which will provide the classification of these algebras only
up to an isomorphism. Computing both classifying objects is the
key step in a recursive algorithm for describing and classifying
the new class of algebras defined as follows:

\begin{definition} \delabel{coflaglbz}
Let $A$ be an algebra and $E$ a vector space. A unital associative
algebra structure $\cdot_E$ on $E$ is called a \emph{co-flag
algebra over $A$} if there exists a positive integer $n$ and a
finite chain of surjective morphisms of algebras
\begin{equation} \eqlabel{lant}
A_n : = (E, \cdot_E) \stackrel{\pi_{n}}{\longrightarrow} A_{n-1}
\stackrel{\pi_{n-1}}{\longrightarrow} A_{n-2} \, \cdots \,
\stackrel{\pi_{2}}{\longrightarrow} A_1 \stackrel{\pi_{1}}
{\longrightarrow} A_{0} := A
\end{equation}
such that ${\rm dim}_k ( {\rm Ker} (\pi_{i}) ) = 1$, for all $i =
1, \cdots, n$. A finite dimensional algebra is called a
\emph{co-flag algebra} if it is a co-flag algebra over the unital
algebra $k$.
\end{definition}

By applying successively \prref{hocechiv} we obtain that a co-flag
algebra over an algebra $A$ is isomorphic to an iteration of
Hochschild products of the form $\bigl(\cdots \bigl( (A \star
k)\star k \bigl) \star \cdots \star k\bigl)$, where the
$1$-dimensional vector space $k$ appears $n$ times in the above
product. The tools used for describing co-flag algebras are the
following:

\begin{definition} \delabel{coflag}
Let $A$ be an algebra. A \emph{co-flag datum of the first kind of
$A$} is a triple $(\lambda, \Lambda, \vartheta)$ consisting of two
algebra maps\footnote{Recall that we assume the algebra maps
$\lambda: A \to k$ to be unit preserving, i.e. $\lambda (1_A) =
1$.} $\lambda$, $\Lambda : A \to k$ and a bilinear map $\vartheta
: A \times A \to k$ satisfying the following compatibilities for
any $a$, $b$, $c\in A$:
\begin{eqnarray}
\vartheta (a, 1_A) = \vartheta (1_A, a) = 0, \qquad \vartheta (a,
bc) - \vartheta (ab, c) = \vartheta ( a, b) \Lambda (c) -
\vartheta (b, c) \lambda (a) \eqlabel{ciudat1}
\end{eqnarray}
A \emph{co-flag datum of the second kind of $A$} is a pair
$(\lambda, u)$ consisting of a linear map $\lambda : A \to k$ such
that $\lambda (1_A) = 1$ and a non-zero scalar $u\in k^*$.
\end{definition}

We denote by ${\mathcal C} {\mathcal F}_1 \, (A)$ (resp.
${\mathcal C} {\mathcal F}_2 \, (A)$) the set of all co-flag data
of the first (resp. second) kind of $A$ and by ${\mathcal C}
{\mathcal F} \, (A) := {\mathcal C} {\mathcal F}_1 \, (A) \,
\sqcup \, {\mathcal C} {\mathcal F}_2 \, (A)$ their coproduct; the
elements of ${\mathcal C} {\mathcal F} \, (A)$  are called
\emph{co-flag data} of $A$. The set of co-flag data ${\mathcal C}
{\mathcal F} \, (A)$ parameterizes the set of all Hochschild
systems of $A$ by a $1$-dimensional vector space. The next result
also describes the first algebra $A_1$ from the exact sequence
\equref{lant} in terms depending only on $A$.

\begin{proposition}\prlabel{cohocflag}
Let $A$ be an algebra. Then there exists a bijection ${\mathcal
H}{\mathcal S} \, (A, \, k) \cong {\mathcal C} {\mathcal F} \,
(A)$ between the set of all Hochschild systems of $A$ by $k$ and
the set of all co-flag data of $A$ given such that the Hochschild
product $A \star k$ associated to $(\lambda, \Lambda, \vartheta)
\in {\mathcal C} {\mathcal F}_1 \, (A)$ is the algebra denoted by
$A_{(\lambda, \Lambda, \vartheta)}$ with the multiplication given
for any $a$, $b \in A$, $x$, $y\in k$ by:
\begin{equation} \eqlabel{patratnul}
(a, x) \star (b, y) = \bigl(ab, \,\, \vartheta (a, b) + \lambda
(a) y + \Lambda (b) x \bigl)
\end{equation}
while the Hochschild product $A \star k$ associated to $(\lambda,
u) \in {\mathcal C} {\mathcal F}_2 \, (A)$ is the algebra denoted
by $A^{(\lambda, u)}$ with the multiplication given for any $a$,
$b \in A$, $x$, $y\in k$ by:
\begin{equation} \eqlabel{patratarbitar}
(a, x) \star (b, y) = \bigl(ab, \,\, u^{-1} \bigl(\lambda(a)
\lambda(b) - \lambda(ab) \bigl) + \lambda (a) y + \lambda (b) x +
u  \, xy \bigl)
\end{equation}
\end{proposition}

\begin{proof}
We have to compute the set of all bilinear maps $\triangleright :
A \times k \to k$, $\triangleleft : k \times A \to k$, $ \vartheta
: A\times A \to k$ and $\cdot \, : k\times k \to k $ satisfying
the compatibility conditions (H0)-(H8) of \prref{hocprod}. Since
$k$ has dimension $1$ there exists a bijection between the set of
all Hochschild datums $\bigl( \triangleright, \, \triangleleft, \,
\vartheta, \, \cdot \bigl)$ of $A$ by $k$ and the set of all
$4$-tuples $(\Lambda, \, \lambda, \, \vartheta, \, u)$ consisting
of two linear maps $\Lambda$, $\lambda: A \to k$, a bilinear map
$\vartheta: A \times A \to k$ and a scalar $u\in k$. The bijection
is given such that the Hochschild datum $\bigl( \triangleright, \,
\triangleleft, \, \vartheta, \, \cdot \bigl)$ corresponding to
$(\Lambda, \, \lambda, \, \vartheta, \, u)$ is defined as follows:
\begin{eqnarray*}
a\triangleright x := \lambda (a) x, \,\,\,\,\, x \triangleleft a
:= \Lambda (a) x, \,\,\,\,\, x \cdot y := u \, x y
\end{eqnarray*}
for all $a \in A$ and $x$, $y\in k$. Now, axiom (H0) holds if and
only if $\vartheta (a, 1_A) = \vartheta (1_A, a) = 0$ and $\lambda
(1_A) = \Lambda (1_A) = 1$. Axioms (H1), (H3), (H4) and (H8) are
trivially fulfilled. Axiom (H5) is equivalent to $\vartheta (a,
bc) - \vartheta (ab, c) = \vartheta ( a, b) \Lambda (c) -
\vartheta (b, c) \lambda (a)$, axiom (H6) is equivalent to
\begin{equation} \eqlabel{faracoc}
\lambda (ab) = \lambda(a) \lambda (b) - u \, \vartheta (a, b)
\end{equation}
while axiom (H7) is equivalent to $\Lambda (ab) = \Lambda(a)
\Lambda (b) - u \, \vartheta (a, b)$. Finally, axiom (H2) is
equivalent to $u \Lambda (a) = u \lambda (a)$, for all $a \in A$.
A discussion on $u$ is imposed by the last compatibility condition
and the conclusion follows easily: ${\mathcal C} {\mathcal F}_1 \,
(A)$ corresponds to the case when $u = 0$ and this will give rise
to the algebras $A \star k = A_{(\lambda, \Lambda, \theta)}$. The
case ${\mathcal C} {\mathcal F}_2 \, (A)$ corresponds to $u \neq
0$; in this case $\Lambda = \lambda$ and the cocycle $\vartheta$
is implemented by $u$ and $\lambda$ via the formula $\vartheta (a,
b) := u^{-1} \bigl(\lambda(a) \lambda(b) - \lambda(ab) \bigl)$,
for all $a$, $b\in A$, that arises from \equref{faracoc}.
Moreover, we can easily check that axiom \equref{ciudat1} is
trivially fulfilled for $\vartheta$ defined as above. The algebra
$A^{(\lambda, u)}$ is just the Hochschild product $A \star k$
associated to this context.
\end{proof}

\begin{remarks} \relabel{nucleunul}
$(1)$ The first family of Hochschild products $A_{(\lambda,
\Lambda, \theta)}$ constructed in \prref{cohocflag} corresponds to
the classical case in which $k \cong 0 \times k$ is a two-sided
ideal of null square in the algebra $A_{(\lambda, \Lambda,
\theta)}$. The algebras $A_{(\lambda, \Lambda, \theta)}$ will be
classified up to an isomorphism in \thref{clasres} below. For the
new families of algebras $A^{(\lambda, u)}$ the kernel of the
canonical projection $\pi_A : A^{(\lambda, u)} \to k$ is equal to
$k \cong 0 \times k$ and this is not a null square ideal since
$(0, 1) \star (0, 1) = (0, \, u) \neq (0, 0)$. Let $(\lambda, u)
\in {\mathcal C} {\mathcal F}_2 \, (A)$ be a co-flag datum of the
second kind of $A$. Taking into account the multiplication on
$A^{(\lambda, \, u)}$ given by \equref{patratarbitar} we can
easily prove that the map:
\begin{equation} \eqlabel{izocaut}
\varphi : A^{(\lambda, \, u)} \to A \times k, \qquad \varphi (a,
x) := (a, \, \lambda (a) + u\, x)
\end{equation}
for all $a\in A$ and $x\in k$ is an isomorphism of algebras (which
does not stabilize $k$, if $u \neq 1$), where $A \times k$ is the
usual direct product of algebras. The inverse of $\varphi$ is
given by $\varphi^{-1} (a, x) = \bigl(a, \, u^{-1} (x - \lambda
(a)\bigl)$, for all $a\in A$ and $x\in k$.
\end{remarks}

We will now describe the algebra $A_{(\lambda, \Lambda,
\vartheta)}$ and $A^{(\lambda, u)}$ by generators and relations.
The elements of $A$ will be seen as elements in $A \times k$ via
the identification $a = (a, \, 0)$ and we denote by $f := (0_A, \,
1) \in A\times k$. Let $\{ e_i \, | \, i \in I \}$ be a basis of
$A$ as a vector space over $k$. Then the algebra $A_{(\lambda,
\Lambda, \vartheta)}$ is the vector space having $\{ f, \, e_i \,
| \, i \in I \}$ as a basis and the multiplication $\star$ given
for any $i\in I$ by:
\begin{equation}\eqlabel{primaalg}
e_i \star e_j = e_i \cdot_A e_j + \vartheta(e_i, \, e_j) \, f,
\,\,\, f^2 = 0, \,\,\, e_i \star f = \lambda(e_i)\,f, \,\,\,
f\star e_i = \Lambda (e_i) \ f
\end{equation}
where $\cdot_A$ denotes the multiplication on $A$. The algebra
$A^{(\lambda, u)}$ is the vector space having $\{ f, \, e_i \, |
\, i \in I \}$ as a basis and the multiplication $\star$ given for
any $i\in I$ by:
\begin{equation}\eqlabel{doiaalg}
e_i \star e_j = e_i \cdot_A e_j + u^{-1} \bigl( \lambda(e_i)
\lambda(e_j) - \lambda(e_i \cdot_A e_j)  \bigl)\, f, \,\,\,\, f^2
= u\, f, \,\,\,\, e_i \star f = f\star e_i = \lambda(e_i)\,f
\end{equation}

Using \prref{cohocflag}, \prref{hocechiv} and the isomorphism
$A^{(\lambda, \, u)} \simeq A \times k$ we obtain:

\begin{corollary}\colabel{descefecti}
Let $A$ be an algebra. A unital associative algebra $B$ has a
surjective algebra map $B \to A\to 0$ whose kernel is
$1$-dimensional if and only if $B$ is isomorphic to $A \times k$
or $A_{(\lambda, \Lambda, \vartheta)}$, for some $(\lambda,
\Lambda, \vartheta) \in {\mathcal C} {\mathcal F}_1 \, (A)$.
\end{corollary}

We are now able to compute the classifying object ${\mathbb G}
{\mathbb H}^{2} \, (A, \, k)$.

\begin{proposition} \prlabel{hoccal1}
Let $A$ be an algebra. Then,
$$
{\mathbb G} {\mathbb H}^{2} \, (A, \, k) \cong \Bigl({\mathcal C}
{\mathcal F}_1 \, (A)/ \approx_1 \Bigl)\,\, \sqcup \,\, k^*
$$
where $\approx_1$ is the following equivalence relation on
${\mathcal C} {\mathcal F}_1 \, (A)$: $(\lambda, \Lambda,
\vartheta) \approx_1 (\lambda', \Lambda', \vartheta')$ if and only
if $\lambda = \lambda'$, $\Lambda = \Lambda'$  and there exists a
linear map $t: A \to k$ such that for any $a$, $b\in A$:
\begin{equation} \eqlabel{primaechiv}
\vartheta (a, b) = \vartheta'(a, b) - t(ab) + \lambda'(a) t(b) +
\Lambda'(b) t(a)
\end{equation}
\end{proposition}

\begin{proof}
It follows from \thref{main1222} and \prref{cohocflag} that
$$
{\mathbb G} {\mathbb H}^{2} \, (A, \, k) \cong \Bigl({\mathcal C}
{\mathcal F}_1 \, (A)/ \approx_1 \Bigl)\,\, \sqcup \,\,
\Bigl({\mathcal C} {\mathcal F}_2 \, (A)/ \approx_2 \Bigl)
$$
where the equivalence relation $\approx_i$ on ${\mathcal C}
{\mathcal F}_i \, (A)$, for $i = 1, 2$, is just the equivalence
relation $\approx$ from \deref{echiaa} written for the sets
${\mathcal C} {\mathcal F}_i \, (A)$ via the bijection ${\mathcal
H}{\mathcal S} \, (A, \, k) \cong {\mathcal C} {\mathcal F} \,
(A)$ given in \prref{cohocflag}. The equivalence relation
$\approx$ written on the set of all co-flag data of the first kind
takes precisely the form from the statement -- we mention that a
linear map $t$ satisfying \equref{primaechiv} has the property
that $t(1_A) = 0$. The equivalence relation $\approx$ written on
${\mathcal C} {\mathcal F}_2 \, (A)$ takes the following form:
$(\lambda, u) \approx_2 (\lambda', u')$ if and only if $u = u'$
and there exists a linear map $t: A \to k$ such that for any $a\in
A$ we have:
\begin{equation}\eqlabel{primaechivii}
\lambda (a) = \lambda'(a) + t(a)\, u'
\end{equation}
Now, if we fix a unit preserving linear map $\lambda^0 : A \to k$
we obtain that the set $\{ (\lambda^0, \, u) \, | \, u \in k^* \}$
is a system of representatives for the equivalence relation
$\approx_2$ on ${\mathcal C} {\mathcal F}_2 \, (A)$ and hence
${\mathcal C} {\mathcal F}_2 \, (A)/ \approx_2 \,\, \cong \, k^*$,
which finishes the proof.
\end{proof}

The way $\approx_1$ is defined in \prref{hoccal1} indicates the
decomposition of $\Bigl({\mathcal C} {\mathcal F}_1 \, (A)/
\approx_1 \Bigl)$ as follows: for two fixed algebra maps
$(\lambda, \, \Lambda) \in {\rm Alg} (A, \, k)$ we shall denote by
${\rm Z}^2_{(\lambda, \, \Lambda)} \, (A, \, k)$ the set of all
normalized $(\lambda, \, \Lambda)$-cocycles; that is, the set of
all bilinear maps $\vartheta : A \times A \to k$ satisfying the
following compatibilities for any $a$, $b$, $c\in A$:
\begin{eqnarray*}
\vartheta (a, 1_A) = \vartheta (1_A, a) = 0, \qquad \vartheta (a,
bc) - \vartheta (ab, c) = \vartheta ( a, b) \Lambda (c) -
\vartheta (b, c) \lambda (a) \eqlabel{ciudat1a}
\end{eqnarray*}
Two $(\lambda, \, \Lambda)$-cocycles $\vartheta$, $\vartheta' : A
\times A \to k$ are equivalent $\vartheta \approx_{1}^{(\lambda,
\, \Lambda)} \vartheta'$ if and only if there exists a linear map
$t: A \to k$ such that
\begin{equation} \eqlabel{ecfolo}
\vartheta (a, b) = \vartheta'(a, b) - t(ab) + \lambda(a) t(b) +
\Lambda(b) t(a)
\end{equation}
for all $a$, $b\in A$. If we denote ${\rm H}^2_{(\lambda, \,
\Lambda)} \, (A, \, k) := {\rm Z}^2_{(\lambda, \, \Lambda)} \, (A,
\, k)/ \approx_{1}^{(\lambda, \, \Lambda)}$ we obtain the
following decomposition of ${\mathbb G} {\mathbb H}^{2} \, (A, \,
k)$:

\begin{corollary} \colabel{hoccal1ab}
Let $A$ be an algebra. Then,
\begin{equation} \eqlabel{astaedefol}
{\mathbb G} {\mathbb H}^{2} \, (A, \, k) \, \cong \, \Bigl(
\sqcup_{(\lambda, \Lambda)} \,\, {\rm H}^2_{(\lambda, \, \Lambda)}
\, (A, \, k) \Bigl) \, \sqcup \, k^*
\end{equation}
where the coproduct on the right hand side is in the category of
sets over all possible algebra maps $\lambda$, $\Lambda: A\to k$.
\end{corollary}

The classifying object ${\mathbb G} {\mathbb H}^{2} \, (A, \, k)$
computed in \coref{hoccal1ab} classifies all Hochschild products
$A \star k$ up to an isomorphism of algebras which stabilizes $k$
and co-stabilizes $A$. In what follows we will consider a less
restrictive classification: we denote by ${\mathbb H} {\mathbb O}
{\mathbb C} \, (A, \, k)$ the set of algebra isomorphism classes
of all Hochschild products $A \star k$. Two cohomologous
Hochschild products $A \star k$ and $A \star' k$ are of course
isomorphic and therefore there exists a canonical projection
${\mathbb G} {\mathbb H}^{2} \, (A, \, k) \twoheadrightarrow
{\mathbb H} {\mathbb O} {\mathbb C} \, (A, \, k)$ between the two
classifying objects. Next we compute ${\mathbb H} {\mathbb O}
{\mathbb C} \, (A, \, k)$.

\begin{theorem} \thlabel{clasres}
Let $A$ be an algebra. Then there exists a bijection:
\begin{equation} \eqlabel{hoculmic}
{\mathbb H} {\mathbb O} {\mathbb C} \, (A, \, k) \cong \Bigl(
{\mathcal C} {\mathcal F}_1 \, (A)/ \equiv \Bigl) \, \sqcup \, \{
\, A\times k \, \}
\end{equation}
where $\equiv$ is the equivalence relation on ${\mathcal C}
{\mathcal F}_1 \, (A)$ defined by: $(\lambda, \Lambda, \vartheta)
\equiv (\lambda', \Lambda', \vartheta')$ if and only if there
exists a triple $(s_0, \, \psi, \, r) \in k^* \times {\rm
Aut}_{\rm Alg} (A) \times {\rm Hom}_k (A, \, k)$ consisting of a
non-zero scalar $s_0 \in k^*$, an algebra automorphism $\psi$ of
$A$ and a linear map $r: A \to k$ such that for any $a$, $b \in A$
we have:
\begin{eqnarray}
&& \lambda = \lambda' \circ \psi, \qquad \Lambda = \Lambda' \circ \psi \eqlabel{2015a} \\
&& \vartheta (a, b) \, s_0 = \vartheta' \bigl(\psi(a), \psi(b)
\bigl) + \lambda(a) r(b) + \Lambda(b) r(a) - r(ab) \eqlabel{2015b}
\end{eqnarray}
\end{theorem}

\begin{proof} \coref{descefecti} shows that any Hochschild product $A
\star k$ is isomorphic to $A_{(\lambda, \Lambda, \vartheta)}$, for
some $(\lambda, \Lambda, \vartheta) \in {\mathcal C} {\mathcal
F}_1 \, (A)$ or to $A^{(\lambda', u')}$, for some $(\lambda', u')
\in {\mathcal C} {\mathcal F}_2 \, (A)$. Since $A^{(\lambda', u')}
\cong A \times k$, the proof relies on the following two steps:

$(1)$ Let $(\lambda, \Lambda, \vartheta)$ and $(\lambda',
\Lambda', \vartheta') \in {\mathcal C} {\mathcal F}_1 \, (A)$.
Then, there exists a bijection between the set of all algebra
isomorphisms $\varphi: A_{(\lambda, \Lambda, \vartheta)} \to
A_{(\lambda', \Lambda', \vartheta')}$ and the set of all triples
$(s_0, \, \psi, \, r) \in k^* \times {\rm Aut}_{\rm Alg} (A)
\times {\rm Hom}_k (A, \, k)$ satisfying the compatibility
conditions \equref{2015a} and \equref{2015b}. The bijection is
given such that the algebra isomorphism $\varphi = \varphi_{(s_0,
\, \psi, \, r)}$ associated to $(s_0, \, \psi, \, r)$ is defined
for any $a\in A$ and $x\in k$ by:
\begin{equation}\eqlabel{morfisass}
\varphi_{(s_0, \, \psi, \, r)} (a, \, x) = \bigl(\psi(a), \, r(a)
+ x s_0 \bigl)
\end{equation}

$(2)$ The algebras $A_{(\lambda, \Lambda, \vartheta)}$ and
$A^{(\lambda', u')} \cong A \times k$ are not isomorphic.

We start by proving $(1)$; although this is more than we need for
proving our theorem, this more general statement will be used
later on in computing the automorphism groups for the algebras
$A_{(\lambda, \Lambda, \vartheta)}$. First we note that there
exists a bijection between the set of all linear maps $\varphi: A
\times k \to A \times k$ and the set of quadruples $(s_0, \,
\beta_0, \, \psi, \, r) \in k\times A \times {\rm Hom}_k (A, \, A)
\times {\rm Hom}_k (A, \, k)$ given such that the linear map
$\varphi = \varphi_{(s_0, \, \beta_0, \, \psi, \, r)}$ associated
to $(s_0, \, \beta_0, \, \psi, \, r)$ is given for any $a\in A$
and $x\in k$ by:
\begin{equation}\eqlabel{4morf}
\varphi (a, \, x) = \bigl( \psi(a) + x\, \beta_0, \, r(a) + x\,
s_0 \bigl)
\end{equation}
We will prove now the following technical fact: a linear map given
by \equref{4morf} is an isomorphism of algebras from $A_{(\lambda,
\Lambda, \vartheta)}$ to $A_{(\lambda', \Lambda', \vartheta')}$ if
and only if $\beta_0 = 0$, $s_0 \neq 0$, $\psi$ is an algebra
automorphism of $A$ and \equref{2015a}-\equref{2015b} hold. Taking
into account the multiplication on $A_{(\lambda, \Lambda,
\vartheta)}$ given by \equref{patratnul}, we can easily obtain
that $\varphi \bigl( (0, x) \star (0, y) \bigl) = \varphi (0, x)
\star' \varphi (0, x)$ if and only if $\beta_0 = 0$, where by
$\star'$ we denote the multiplication of $A_{(\lambda', \Lambda',
\vartheta')}$. Hence, in order for $\varphi$ to be an algebra map
it should take the following simplified form for any $a\in A$ and
$x\in k$:
\begin{equation}\eqlabel{4morfb}
\varphi (a, \, x) = \bigl( \psi(a), \, r(a) + x\, s_0 \bigl)
\end{equation}
for some triple $(s_0, \, \psi, \, r) \in k \times {\rm Hom}_k (A,
\, A) \times {\rm Hom}_k (A, \, k)$. Next we prove that a linear
map given by \equref{4morfb} is an algebra morphism from
$A_{(\lambda, \Lambda, \vartheta)}$ to $A_{(\lambda', \Lambda',
\vartheta')}$ if and only if $\psi: A \to A$ is an algebra map and
the following compatibilities are fulfilled for any $a$, $b\in A$:
\begin{eqnarray}
&& \lambda (a) \, s_0 = \lambda'\bigl(\psi (a) \bigl) \, s_0,
\qquad \Lambda (a) \, s_0 = \Lambda'\bigl(\psi (a) \bigl) \, s_0 \eqlabel{2015cc} \\
&& r(ab) + \vartheta (a, b) \, s_0 = \vartheta' \bigl(\psi(a),
\psi(b) \bigl) + \lambda'(\psi(a)) r(b) + \Lambda'(\psi(b) ) r(a)
\eqlabel{2015dd}
\end{eqnarray}
Indeed, $\varphi$ preserves the unit $(1_A, 0)$ if and only if
$\psi (1_A) = 1_A$ and $r(1_A) = 0$. On the other hand we can
prove that the first (resp. the second) compatibility of
\equref{2015cc} is exactly the condition $\varphi \bigl( (a, 0)
\star (0, x) \bigl) = \varphi (a, 0) \star' \varphi (0, x)$ (resp.
$\varphi \bigl( (0, x) \star (a, 0) \bigl) = \varphi (0, x) \star'
\varphi (a, 0)$). Finally, the condition $\varphi \bigl( (a, 0)
\star (b, 0) \bigl) = \varphi (a, 0) \star' \varphi (b, 0)$ is
equivalent to the fact that $\psi$ is an algebra endomorphism of
$A$ and \equref{2015dd} holds. Finally, the condition $r(1_A) = 0$
follows by taking $a = b = 1_A$ in \equref{2015dd}. Step $(1)$ is
finished if we prove that an algebra map $\varphi = \varphi_{(s_0,
\, \psi, \, r)}$ given by \equref{4morfb} is bijective if and only
if $s_0 \neq 0$ and $\psi$ is an automorphism of $A$. Assume first
that $s_0 \neq 0$ and $\psi$ is bijective with the inverse
$\psi^{-1}$. Then, we can see that $\varphi_{(s_0, \, \psi, \,
r)}$ is an isomorphism of algebras with the inverse given by
$\varphi^{-1}_{(s_0, \, \psi, \, r)} := \varphi_{(s_0^{-1}, \,
\psi^{-1}, \, -(r\circ \psi^{-1}) s_0^{-1})}$. Conversely, assume
that $\varphi$ is bijective. Then its inverse $\varphi^{-1}$ is an
algebra map and thus has the form $\varphi^{-1} (a, \, x) = (\psi'
(a), \, r' (a) + x s_0')$, for some triple $(s_0', r', \psi')$. If
we write $\varphi^{-1} \circ \varphi (0, \, 1) = (0, \, 1)$ we
obtain that $s_0 s_0' = 1$ i.e. $s_0$ is invertible in $k$. In the
same way $\varphi^{-1} \circ \varphi (a, \, 0) = (a, \, 0) =
\varphi \circ \varphi^{-1} (a, \, 0)$ gives that $\psi$ is
bijective and $\psi' = \psi^{-1}$.

We will prove now the assertion from step $(2)$. Assume that
$\varphi : A_{(\lambda, \Lambda, \vartheta)} \to A^{(\lambda',
u')}$ is an algebra map. Thus, $\varphi$ is given by
\equref{4morf}, for some quadruple $(s_0, \, \beta_0, \, \psi, \,
r)$. Now, we can see that the algebra map condition $\varphi
\bigl( (0, x) \star (0, y) \bigl) = \varphi (0, x) \star' \varphi
(0, y)$ holds if and only if $\beta_0 = 0$ and $s_0 = 0$, where
$\star'$ denotes the multiplication on the algebra $A^{(\lambda',
u')}$. Hence, $\varphi$ takes the form $\varphi (a, \, x) = \bigl(
\psi(a), \, r(a) \bigl)$, for all $a\in A$ and $x\in k$. Such a
map is never an isomorphism of algebras since is not injective and
thus two algebras of the form $A_{(\lambda, \Lambda, \vartheta)}$
and $A^{(\lambda', u')}$ are never isomorphic. The theorem is now
completely proved.
\end{proof}

\begin{remark} \relabel{homoth}
The compatibility condition \equref{2015b} of \thref{clasres}
highlights the difficulty of classifying co-flag algebras over a
given algebra $A$: it generalizes the classical
Kroneker-Williamson equivalence of bilinear forms whose
classification was started in \cite{will} and finished in
\cite{horn} over algebraically closed fields. We recall that two
bilinear forms $\vartheta$ and $\vartheta'$ on a vector space $A$
are called \emph{isometric} if there exists a linear automorphism
$\psi \in {\rm Aut}_k (A)$ such that $\vartheta (x, y) =
\vartheta'( \psi (x), \, \psi(y))$, for all $x$, $y\in A$. If the
cocycles $\vartheta$ and $\vartheta'$ are isometric as bilinear
forms on $A$ and $\psi$ is an algebra automorphism of $A$, then
\equref{2015b} holds by taking $s_0 := 1$ and $r := 0$, the
trivial map. For future references to the problem of classifying
bilinear forms up to an isometry we refer to \cite{horn}.
\end{remark}

\thref{clasres} can be applied to classify all semidirect products
of algebras of the form $A \# k$. We recall from \exref{abelian}
that a semidirect product $A \# k$ is just a Hochschild product
$A_{(\lambda, \Lambda, \vartheta)} = A\star k$ having a trivial
cocycle. The algebra obtained in this way will be denoted by
$A_{(\lambda, \, \Lambda)}$. Directly from the proof of
\thref{clasres} we obtain:

\begin{corollary}\colabel{cassemidirect}
Let $A$ be an algebra, $(\lambda, \, \Lambda)$ and $(\lambda', \,
\Lambda')$ two pairs consisting of algebra maps from $A$ to $k$.
Then there exists an isomorphism of algebras $A_{(\lambda, \,
\Lambda)} \cong A_{(\lambda', \, \Lambda')}$ if and only if there
exists $\psi \in {\rm Aut}_{\rm Alg} (A)$ such that $\lambda =
\lambda' \circ \psi$ and $\Lambda = \Lambda' \circ \psi$.
\end{corollary}

An interesting special case occurs for the algebras $A$ such that
there is no algebra map $A \to k$ (e.g. the classical Weyl algebra
$W_1 (k) = k < \, x, \, y \, | \, xy - yx = 1 \, >$ or the matrix
algebra ${\rm M}_n (k)$, for $n\geq 2$). Using \prref{hoccal1} and
\thref{clasres} we obtain:

\begin{corollary}\colabel{casspecial}
Let $A$ be an algebra for which there is no algebra map $A \to k$.
Then
$$
{\mathbb G} {\mathbb H}^{2} \, (A, \, k) \cong k^*, \qquad
{\mathbb H} {\mathbb O} {\mathbb C} \, (A, \, k) = \{ \, A\times k
\, \}
$$
In particular, up to an isomorphism, the only algebra $B$ for
which there exists a surjective algebra map $B \to A$ having a
$1$-dimensional kernel is the direct product $A \times k$.
\end{corollary}

Determining the automorphism group of a given algebra is an old
and very difficult problem, intensively studied in invariant
theory (see \cite{cek} and their references). As already
mentioned, the first step proved in the proof of \thref{clasres}
allows us to compute the automorphism group ${\rm Aut}_{\rm Alg}
\, (A_{(\lambda, \Lambda, \vartheta)})$, for any $(\lambda,
\Lambda, \vartheta) \in {\mathcal C} {\mathcal F}_1 \, (A)$. Let
$k^*$ be the units group of $k$, $k^* \times {\rm Aut}_{\rm Alg}
(A)$ the direct product of groups and $(A^*, +)$ the underlying
abelian group of the linear dual $A^* = {\rm Hom}_k (A, \, k)$.
Then the map given for any $s_0 \in k^*$, $\psi \in {\rm Aut}_{\rm
Alg} (A)$ and $r\in A^*$ by:
$$
\zeta: k^* \times {\rm Aut}_{\rm Alg} (A) \to {\rm Aut}_{\rm Gr}
\, (A^*, +), \qquad \zeta (s_0, \, \psi) \, (r) := s_0^{-1} \, r
\circ \psi
$$
is a morphism of groups. Thus, we can construct the semidirect
product of groups $A^* \, \ltimes_{\zeta} \bigl(k^* \times {\rm
Aut}_{\rm Alg} (A) \bigl)$ associated to $\zeta$. The next result
shows that ${\rm Aut}_{\rm Alg} (A_{(\lambda, \Lambda, \vartheta)}
)$ is isomorphic to a certain subgroup of the semidirect product
$A^* \, \ltimes_{\zeta} \bigl(k^* \times {\rm Aut}_{\rm Alg} (A)
\bigl)$.

\begin{corollary} \colabel{izoaut}
Let $A$ be an algebra, $(\lambda, \Lambda, \vartheta) \in
{\mathcal C} {\mathcal F}_1 \, (A)$ a co-flag datum of the first
kind of $A$ and let ${\mathcal G} \bigl(A, \, (\lambda, \Lambda,
\vartheta) \bigl)$ be the set of all triples $(s_0, \, \psi, \, r)
\in k^* \times {\rm Aut}_{\rm Alg} (A) \times A^*$ such that for
any $a$, $b\in A$:
$$
\lambda = \lambda \circ \psi, \quad \Lambda = \Lambda \circ \psi,
\quad \vartheta (a, b) \, s_0 = \vartheta \bigl(\psi(a), \psi(b)
\bigl) + \lambda(a) r(b) + \Lambda(b) r(a) - r(ab)
$$
Then, there exists an isomorphism of groups ${\rm Aut}_{\rm Alg} (
A_{(\lambda, \Lambda, \vartheta)} ) \cong {\mathcal G} \,
\bigl(A, \, (\lambda, \Lambda, \vartheta) \bigl)$, where
${\mathcal G} \,  \bigl(A, \, (\lambda, \Lambda, \vartheta)
\bigl)$ is a group with respect to the following multiplication:
\begin{equation} \eqlabel{graut}
(s_0, \, \psi, \, r) \cdot (s_0', \, \psi', \, r') := (s_0 s_0',
\, \psi\circ \psi', \, r \circ \psi' + s_0 r' )
\end{equation}
for all $(s_0, \, \psi, \, r)$, $(s_0', \, \psi', \, r') \in \in
{\mathcal G} \, \bigl(A, \, (\lambda, \Lambda, \vartheta) \bigl)$.
Moreover, the canonical map
$$
{\mathcal G} \, \bigl(A, \, (\lambda, \Lambda, \vartheta) \bigl)
\longrightarrow A^* \, \ltimes_{\zeta} \bigl(k^* \times {\rm
Aut}_{\rm Alg} (A) \bigl), \qquad (s_0, \, \psi, \, r) \mapsto
\bigl (s_0^{-1} r, \, (s_0, \, \psi) \bigl)
$$
in an injective morphism of groups.
\end{corollary}

\begin{proof} The fact that ${\mathcal G} \,  \bigl(A, \, (\lambda, \Lambda, \vartheta)
\bigl)$ is a group with respect to the multiplication
\equref{graut} follows by a straightforward computation which is
left to the reader: the unit is $(1, \, {\rm Id}_A, \, 0)$ and the
inverse of $(s_0, \, \psi, \, r)$ is $(s_0^{-1}, \, \psi^{-1}, \,
- s_0^{-1} (r\circ \psi^{-1}))$. The first statement follows from
the proof of \thref{clasres}, step $(1)$, since $\varphi_{(s_0,
\psi, r)} \circ \varphi_{(s_0', \psi', r')} = \varphi_{(s_0 s_0',
\, \psi \circ \psi', \, r\circ \psi' + s_0 r')}$, where
$\varphi_{(s_0, \psi, r)}$ is an automorphism of $A_{(\lambda,
\Lambda, \vartheta)}$ given by \equref{morfisass}. The last
assertion follows by a routine computation.
\end{proof}

Now we shall provide some explicit examples. The first example
shows the limitations of the classical approach for the extension
problem: there is no $(1 + n^2)$-dimensional algebra with an
algebra projection on the matrix algebra ${\rm M}_n(k)$ whose
kernel is a null square ideal, but there exists a family of $(1 +
n^2)$-dimensional algebras which project on the matrix algebra
${\rm M}_n(k)$. We denote by $\{e_{ij} \, | \, i, j = 1, \cdots, n
\}$ the canonical basis of ${\rm M}_n(k)$, i.e. $e_{ij}$ is the
matrix having $1$ in the $(i, j)^{th}$ position and zeros
elsewhere while $\delta^k_j$ and $\delta^{(i, j)}_{(n, n)}$ denote
the Kroneker symbols.

\begin{example}\exlabel{matrici}
Let $n \geq 2$ be a positive integer. Then, ${\mathbb G} {\mathbb
H}^{2} \, ({\rm M}_n(k), \, k) \cong  k^* $ and the equivalence
classes of all $(1 + n^2)$-dimensional algebras with an algebra
projection on ${\rm M}_n(k)$ are the following algebras denoted by
${\rm M}_n(k)^u$ and defined for any $u\in k^*$ as the vector
space having $\{f, \, e_{ij} \, | \, i, j = 1, \cdots, n \}$ as a
basis and the multiplication given for any $i$, $j = 1, \cdots, n$
by:
$$
f^2 := u \, f, \quad e_{ij} \star f = f \star e_{ij} :=
\delta^{(i, j)}_{(n, n)} \, f, \quad e_{ij} \star e_{kl} :=
\delta^j_k \, e_{il} + u^{-1} \bigl( \delta^{(i, j)}_{(n, n)}
\delta^{(k, l)}_{(n, n)} - \delta^k_j \, \delta^{(i, l)}_{(n,
n)}\bigl) \, f
$$
Furthermore, ${\mathbb H} {\mathbb O} {\mathbb C} \, ( {\rm
M}_n(k), \, k) = \{ {\rm M}_n(k) \times k \}$.

The result follows by applying \prref{hoccal1} and
\coref{casspecial} since there is no unitary algebra map ${\rm
M}_n (k) \to k$. If we consider $\lambda^0 : {\rm M}_n(k) \to k$
defined by $\lambda^0 (e_{ij}) := \delta^{(i, j)}_{(n, n)}$, for
all $i$, $j = 1, \cdots, n$, then $\{ (\lambda^0, \, u) \, | \, u
\in k^* \}$ is a system of representatives for the equivalence
relation $\approx_2$. The algebra ${\rm M}_n(k)^u$ associated to
the pair $(\lambda^0, \, u)$ is the vector space having $\{f, \,
e_{ij} \, | \, i, j = 1, \cdots, n \}$ as a basis while the
multiplication given by \equref{doiaalg} comes down to the one in
the statement.
\end{example}

An interesting example through the subtle arithmetics involved in
the classification of the corresponding Hochschild products is the
group algebra $k[C_n]$, where for a positive integer $n \geq 2$ we
denote by $C_n$ the cyclic group of order $n$ generated by $d$. We
introduce the following notation: for any $i$, $j = 1, \cdots, n -
1$ we shall denote by $i \ast j$ the positive integer given by
$$
i \ast j := \left\{
\begin{array}{lcl} i+j & \mbox{if}& j+i < n\\ i+j-n & \mbox{if}& j+i \geq n
\end{array}\right.
$$
In what follows $U_{n}(k) := \{\omega \in k ~|~ \omega^{n} = 1\}$
denotes the cyclic group of $n$-th roots of unity in $k$ and
$\mathcal{A}(n, \, k) := \{x \in U (k[C_{n}]) ~|~ \psi: k[C_{n}]
\to k[C_{n}], \, \psi(d^{i}) = x^{i}, \,\, i = 0, 1, \cdots, n-1,
\,\, {\rm is \,\, an \,\, algebra\,\, automorphism} \}$.

\begin{example}\exlabel{grupala}
Let $k$ be a field such that $n$ is invertible in $k$. Then:
$$
{\mathbb G} {\mathbb H}^{2} \, (k[C_n], \, k) \cong \bigl(U_{n}(k)
\times U_{n}(k)\bigl) \, \sqcup \, k^*
$$
and the equivalence classes of $(n + 1)$-dimensional algebras with
an algebra projection on $k[C_n]$ are the families of algebras
having $\{f, \, d^{i} ~|~ i = 1, \cdots,  n\}$ as a basis over $k$
and the multiplication $\star$ defined for any $(\alpha, \, \beta)
\in U_{n}(k) \times U_{n}(k)$, $u \in k^{*}$ and $i$, $j = 1,
\cdots, n$ by:
\begin{eqnarray*}
k[C_{n}]_{(\alpha, \, \beta)}:&& \,\,\,d^{i} \star d^{j} =
d^{i+j}, \,\,\,\,\, f^{2} = 0, \,\,\,\,\, d^{i} \star f =
\alpha^{i} \, f,
\,\,\,\,\, f\star d^{i} = \beta^{i} \, f\\
k[C_{n}]^{u}:&& \,\,\, d^{i} \star d^{j} = d^{i+j} +
u^{-1}(\delta^n_i \, \delta^n_j - \delta^n_{i+j})f, \,\,\,\,\,
f^{2} = uf, \,\,\,\,\, d^{i} \star f = f \star d^{i} =
\delta^n_i\,f
\end{eqnarray*}
Furthermore, there exists a bijection
$$
{\mathbb H}{\mathbb O} {\mathbb C} \, ( k[C_n], \, k) \cong
\bigl(U_{n}(k) \times U_{n}(k)/\equiv \bigl) \, \sqcup \, \{k[C_n]
\times k\}
$$
where $\equiv$ is the following equivalence relation on $U_{n}(k)
\times U_{n}(k)$: two pairs $(\alpha, \, \beta)$, $(\alpha ', \,
\beta ')$ of $n$-th roots of unity in $k$ are equivalent $(\alpha,
\, \beta) \equiv (\alpha ', \, \beta ')$ if and only if there
exists $x_{0} + x_{1} \, d + \cdots + x_{n-1} \, d^{n-1} \in
\mathcal{A}(n, \, k)$ such that
\begin{eqnarray} \eqlabel{bizaaar}
\alpha ' = x_{0} + x_{1} \, \alpha + \cdots + x_{n-1} \,
\alpha^{n-1}, \qquad \beta ' = x_{0} + x_{1} \, \beta + \cdots +
x_{n-1} \, \beta^{n-1}
\end{eqnarray}

To start with we point out that the algebra maps $k[C_{n}] \to k$
are parameterized by the cyclic group of $n$-th roots of unity in
$k$. Consider $\alpha$, $\beta \in U_{n}(k)$ which implement
$\lambda$ and respectively $\Lambda$, i.e. $\lambda(d) = \alpha$
and $\Lambda(d) = \beta$. We are left to compute the set of all
$(\lambda, \, \Lambda)$-cocycles. To this end we denote $\vartheta
(d^{i}, \, d) := \xi_{i}$, $i = 1, \cdots, n-1$ and we will see
that these elements will allow us to completely determine the
cocycle $\vartheta: k[C_n] \times k[C_n] \to k$. Indeed, by
writing down \equref{ciudat1} for triples of the form $(d^{i}, \,
d^{j}, \, d)$ and using induction we obtain the following general
formula:
$$
\vartheta(d^{i},\, d^{j}) = \sum_{k=0}^{j-1}\, \xi_{i \ast k} \,
\beta^{j-1-k} - \bigl(\sum_{l=1}^{j-1} \, \xi_{l} \,
\beta^{j-1-l}\bigl) \alpha^{i}
$$
for all $i$, $j = 1, \cdots, n-1$, where $\xi_{0} := 0$.
Furthermore, by writing down \equref{ciudat1} for triples of the
form $(d^{i}, \, d^{n-i}, \, d^{i})$ and using the above formula
for $\vartheta$ we obtain the following compatibility which needs
to be fulfilled for any $i = 1, \cdots, n-1$:
$$
(\alpha^{i} - \beta^{i}) (\xi_{n-1} + \xi_{n-2} \, \beta + ... +
\xi_{1} \, \beta^{n-2}) = 0
$$
Therefore we distinguish two cases, namely: $\alpha = \beta$ or
$\alpha \neq \beta$ and $\xi_{n-1} + \xi_{n-2} \, \beta + ... +
\xi_{1} \, \beta^{n-2} = 0$. It follows that ${\mathcal C}
{\mathcal F}_1 \, (k[C_{n}]) \cong (U_{n}(k) \times k^{n-1}) \cup
\{(\alpha, \, \beta, \, \xi_{1}, \, \xi_{2}, \, ..., \, \xi_{n-2})
\in U_{n}(k) \times U_{n}(k) \times k^{n-2} ~|~ \alpha \neq
\beta\}$ and the bijection associates to any $(\alpha, \,
\delta_{1}, \, \delta_{2}, \, ..., \, \delta_{n-1}) \in U_{n}(k)
\times k^{n-1}$ the co-flag datum of the first kind
$(\lambda_{\alpha},\, \Lambda_{\alpha},\, \vartheta_{\delta})$
given for all $i$, $j = 1, \cdots, n-1$ by:
$$
\lambda(d) = \Lambda(d) := \alpha,\,\,\,
\vartheta_{\delta}(d^{i},\, d^{j}) := \sum_{k=0}^{j-1}\, \delta_{i
\ast k} \, \alpha^{j-1-k} - \bigl(\sum_{l=1}^{j-1} \, \delta_{l}
\, \alpha^{i+j-1-l}\bigl)
$$
where $\delta_{0} = 0$, and to any $(\beta,\, \gamma,\, \xi_{1},
\, \xi_{2}, \, ...,\, \xi_{n-2}) \in U_{n}(k) \times U_{n}(k)
\times k^{n-2}$, with $\beta \neq \gamma$, associates the co-flag
datum of the first kind $(\overline{\lambda}_{\beta},\,
\overline{\Lambda}_{\gamma},\, \overline{\vartheta}_{\xi})$ given
for any $i$, $j = 1, \cdots, n-1$ by:
$$
\overline{\lambda}_{\beta}(d) := \beta,\,\,\,
\overline{\Lambda}_{\gamma}(d) := \gamma, \,\,\,
\overline{\vartheta}_{\xi}(d^{i},\, d^{j}) := \sum_{k=0}^{j-1}\,
\xi_{i \ast k} \, \gamma^{j-1-k} - \bigl(\sum_{l=1}^{j-1} \,
\xi_{l} \, \gamma^{j-1-l}\bigl) \beta^{i}
$$
where $\xi_{0} = 0$ and $\xi_{n-1} = - \bigl(\xi_{n-2} \, \beta
+\, ...\, + \xi_{1}\, \beta^{n-2}\bigl)$. It is now obvious that a
$(\lambda_{\alpha},\, \Lambda_{\alpha})$-cocycle is never
equivalent to a $(\overline{\lambda}_{\beta},\,
\overline{\Lambda}_{\gamma})$-cocycle for any $\alpha$, $\beta$,
$\gamma \in U_{n}(k)$, $\beta \neq \gamma$. Furthermore, by a
rather long but straightforward computation it can be easily seen
that for all $\alpha \in U_{n}(k)$, any $(\lambda_{\alpha},\,
\Lambda_{\alpha})$-cocycle, say $\vartheta_{\delta}$, is
equivalent (in the sense of \equref{ecfolo}) to the trivial
cocycle via the linear map $t: k[C_{n}] \to k$ defined by $t(1) :=
0$ and for any $i = 2, \cdots, n-1$:
\begin{eqnarray*}
t(d) := n^{-1} \sum_{j=1}^{n-1} \alpha^{j} \delta_{n-j}, \quad
t(d^{i}) := n^{-1} i \alpha^{i-1} \sum_{j=1}^{n-1} \alpha^{j}
\delta_{n-j} - \sum_{j=0}^{i-2} \alpha^{j} \delta_{i-1-j}
\end{eqnarray*}
Therefore we have $\sqcup_{(\lambda_{\alpha}, \Lambda_{\alpha})}
\,\, {\rm H}^2_{(\lambda_{\alpha}, \, \Lambda_{\alpha})} \,
(k[C_{n}], \, k) \cong \{(\alpha, \, \alpha) ~|~ \alpha \in
U_{n}(k)\}$. A similar statement holds for the second family of
co-flag data of the first kind: for all $\beta$, $\gamma \in
U_{n}(k)$, with $\beta \neq \gamma$, any
$(\overline{\lambda}_{\beta},\,
\overline{\Lambda}_{\gamma})$-cocycle, say
$\overline{\vartheta}_{\xi}$, is equivalent to the trivial cocycle
via the linear map $\overline{t}: k[C_{n}] \to k$ defined by
$\overline{t}(1) := \overline{t}(d) := 0$ and $\overline{t}(d^{i})
:= - \sum_{j=0}^{i-2} \xi_{i-1-j}\, \gamma^{j}$, for all $i = 2,
\cdots, n-1$. Thus $\sqcup_{(\overline{\lambda}_{\beta},
\overline{\Lambda}_{\gamma})} \,\, {\rm
H}^2_{(\overline{\lambda}_{\beta}, \,
\overline{\Lambda}_{\gamma})} \, (k[C_{n}], \, k) \cong \{(\alpha,
\, \beta) \in U_{n}(k) \times U_{n}(k) ~|~ \alpha \neq \beta\}$.

Therefore, we have proved that $\Bigl({\mathcal C} {\mathcal F}_1
\, (k[C_n])/ \approx_1 \Bigl) \, \cong \, U_{n}(k) \times
U_{n}(k)$ and the corresponding co-flag algebras are those denoted
by $k[C_{n}]_{(\alpha, \, \beta)}$. For the co-flag data of the
second kind of $k[C_n]$ we choose the set of representatives
$\{(\lambda^{0}, u)~|~ u \in k^{*}\}$ for the equivalence relation
$\approx_{2}$, where $\lambda^{0}: k[C_{n}] \to k$ is given by
$\lambda^{0}(d^{i}) = \delta^n_i$, for all $i = 1, \cdots, n$.
They give rise to the co-flag algebras denoted by $k[C_{n}]^{u}$.
The conclusion now follows from \coref{hoccal1ab}. Finally, the
assertion regarding ${\mathbb H} {\mathbb O} {\mathbb C} \, (
k[C_n], \, k)$ is an easy consequence of \coref{cassemidirect}.
\end{example}

\begin{remark} \relabel{angel}
\exref{grupala} shows that any Hochschild product $k[C_n] \star k$
is isomorphic either to the direct product $k[C_n] \times k$, or
to a semi-direct product $k[C_n]_{(\alpha, \, \beta)}$,
parameterized by the group $U_{n}(k) \times U_{n}(k)$. The
explicit description of the isomorphism classes of the algebras
$k[C_n]_{(\alpha, \, \beta)}$ indicated by the equivalence
relation \equref{bizaaar} is a difficult number theory problem
which relies heavily on the arithmetics of the positive integer
$n$ as well as on the base field $k$. Furthermore, the problem is
also related to other two intensively studied problems in the
theory of group algebras, namely the description of all invertible
elements and the automorphism group of a group algebra \cite{jan,
mili, oli}. Indeed, the key set $\mathcal{A}(n, \, k)$ which
appears in the description of the classifying object ${\mathbb
H}{\mathbb O} {\mathbb C} \, ( k[C_n], \, k)$ parameterizes in
fact the automorphism group ${\rm Aut}_{\rm Alg} (k[C_n])$. Any
element of $\mathcal{A}(n, \, k)$ is invertible in $k[C_n]$ and
has order $n$ in the group $U (k[C_n])$. These elements depend
essentially on $n$ and the base field $k$. Indeed, let $X^n - 1 =
f_1 f_2 \cdots f_t$ be the decomposition of $X^n -1$ as a product
of irreducible polynomials in $k[X]$. If we denote by
$\varepsilon_i$ a root of $f_i$ in a fixed algebraic closure of
$k$, we obtain that there exists a canonical isomorphism of
$k$-algebras $k[C_n] \cong k(\varepsilon_1) \times \cdots \times
k(\varepsilon_t)$ that maps the generator $d$ of $C_n$ to
$(\varepsilon_1, \cdots, \varepsilon_t)$. Thus, ${\rm Aut}_{\rm
Alg} (k[C_n])$ is isomorphic to a direct product between all
wreath product of ${\rm Aut} \bigl(k(\varepsilon_i)\bigl)$ and the
symmetric groups \cite{oli}.
\end{remark}

Applying \exref{grupala} for $n = 2$, we obtain the classification
of all $3$-dimensional algebras with an algebra projection on
$k[C_2] \cong k\times k$.

\begin{example} \exlabel{grup2}
If $k$ a field of characteristic $\neq 2$, then:
\begin{eqnarray}
{\mathbb G} {\mathbb H}^{2} \, (k[C_2], \, k) &\cong & \bigl(
\{\pm
1\} \times \{\pm 1\} \bigl) \, \sqcup \, k^* \\
{\mathbb H}{\mathbb O} {\mathbb C} \, ( k[C_2], \, k) &\cong &
\{k^3, \,\, k[X, Y]/ (X^2 -1, \, Y^2, \, XY - Y), \,\, A_{21} \}
\eqlabel{coflag322}
\end{eqnarray}
where $A_{21}$ is the $3$-dimensional non-commutative algebra
having $\{1, d, f\}$ as a basis and the multiplication given by
$d^2 = 1$, $f^2 = 0$, $df = - fd = f$.
\end{example}

Now we highlight the efficiency of our methods in order to
classify co-flag algebras of a given dimension. If $k$ is a field
of characteristic $\neq 2$, then, up to an isomorphism, there
exists only two co-flag algebras of dimension $2$: the algebras
$k[X]/(X^2)$ and $k[C_2] \cong k \times k$ \cite[Corollary
4.5]{am-2014}. If $k\neq k^2$, we mention that the other family of
$2$-dimensional algebras, namely the quadratic field extension
$k(\sqrt{d})$, for some $d \in k \setminus k^2$ does not contain
co-flag algebras since there is no algebra map $k(\sqrt{d}) \to
k$. The co-flag algebras over $k[C_2]$ are classified by
\equref{coflag322} and thus, in order to classify all
$3$-dimensional co-flag algebras we need to classify the co-flag
algebras over $k[X]/(X^2)$.

\begin{example}\exlabel{unexem}
Let $A := k[X]/(X^2)$. Then ${\mathbb G} {\mathbb H}^{2} \,
(k[X]/(X^2), \, k) \cong k \, \sqcup \, k^*$ and the equivalence
classes of $3$-dimensional algebras that have an algebra
projection on $k[X]/(X^2)$ are two families of algebras defined
for any $a\in k$ and $u\in k^*$ as follows:
$$
A_a := k[X, \, Y] / (X^2 - a\,Y, \, Y^2, \, XY), \qquad A^u :=
k[X, \, Y] / (X^2, \, Y^2 - u\, Y, \, XY)
$$
Furthermore, ${\mathbb H} {\mathbb O} {\mathbb C} \, ( k[X]/(X^2),
\, k) = \{ A_0, \, A_1, \, A^1 \}$, i.e. up to an isomorphism
there exist three co-flag algebras of dimension $3$ over
$k[X]/(X^2)$.

Indeed, $A$ is the $2$-dimensional algebra having $1$ and $x$ as a
basis and $x^2 = 0$. Thus $A$ has only one algebra map $A \to k$,
namely the one sending $x$ to $0$. Hence, there exists a bijection
${\mathcal C} {\mathcal F}_1 \, (A) \cong k$ such that the co-flag
datum of the first kind $(\lambda, \Lambda, \vartheta)$ associated
to $a \in k$ is given by
$$
\vartheta (x, x) := a, \quad \lambda(x) = \Lambda(x) = \vartheta
(1, x) = \vartheta (x, 1) = \vartheta (1, 1) := 0
$$
We can easily see that the equivalence relation $\approx_1$ of
\prref{hoccal1} becomes equality, i.e. $a \approx_1 a'$ if and
only if $a = a'$ and hence ${\mathcal C} {\mathcal F}_1 \,
(A)/\approx_1  \, \cong \, k$. The families of algebras associated
to such a co-flag datum of the first kind as defined by
\equref{primaalg} are the $3$-dimensional algebras having $\{ f,
\, 1, \, x \}$ as a basis and the multiplication given by: $x\star
x = a f$, $f^2 = x\star f = f\star x = 0$, which is the algebra
$A_a$. For the last part we apply \prref{hoccal1} which proves
that ${\mathcal C} {\mathcal F}_2 \, (A)/ \approx_2 \, \cong k^*$:
the algebra $A^u$, for all $u\in k^*$, is precisely the algebra
defined by \equref{doiaalg} associated to the co-flag datum of the
second king $(\lambda^0, u)$, where $\lambda^0$ is the linear map
given by $\lambda^0 (x) := 0$, $\lambda^0 (1) := 1$. The last
statement follows from \thref{clasres} or it can be proved
directly as follows: we observe that, for any $u\in k^*$, there
exists and isomorphism of algebras $A^u \cong A^1 = k[X, \, Y] /
(X^2, \, Y^2 - Y, \, XY)$. On the other hand, there exists an
isomorphism of algebras $A_a \cong A_1 = k[X, \, Y] / (X^2 - Y, \,
Y^2, \, XY)$, for all $a\in k^*$ and any two algebras $A_0$, $A_1$
and $A^1$ are not isomorphic to each other.
\end{example}

To conclude, using \exref{grup2} and \exref{unexem} we obtain:

\begin{corollary} \colabel{clasfinal}
If $k$ is a field of characteristic $\neq 2$ then, up to an
isomorphism, there exist exactly six $3$-dimensional co-flag
algebras namely:
\begin{eqnarray*}
&& k^3, \quad k[X, Y]/(X^2 -1, Y^2, XY - Y), \quad k < x, \, y \, | \, x^2 = 1, \, y^2 = 0, \, xy = -yx = y > \\
&& k[X, \, Y] / (X^2, \, Y^2, \, XY), \,\,\, k[X, \, Y] / (X^2 -
Y, \, Y^2, \, XY), \,\,\, k[X, \, Y] / (X^2, \, Y^2 - Y, \, XY)
\end{eqnarray*}
\end{corollary}

In particular, if $k := \CC$ the field of complex numbers,
\coref{clasfinal} shows that only 6 out of the 22 types of
algebras of dimension $3$ are co-flag algebras. Moreover, we also
highlight the efficiency of \thref{clasres} in classifying co-flag
algebras by turning the problem into a purely computational one
using a recursive method: if we consider $A$ to be each of the
algebras from \coref{clasfinal} and using the results of this
section we will arrive at the classification of $4$-dimensional
co-flag algebras. Of course the difficulty of the computations
increases along side with the dimension.

Very interesting and completely different from ${\rm M}_n(k)$ is
the case when $A := \mathcal{T}_{n}(k)$ is the algebra of upper
triangular matrices, i.e. $\mathcal{T}_{n}(k)$ is the subalgebra
of ${\rm M}_n(k)$ having $B := \{e_{ij} ~|~ i,\,j = 1,\, 2,\,
\cdots,\, n,\, i\leq j\}$ as the canonical basis over $k$. In
order to write down the classifying object ${\mathbb G} {\mathbb
H}^{2} \, (\mathcal{T}_{n}(k), \, k)$ we introduce the following
three sets of matrices of trace zero, defined for any $u$, $v$, $w
= 1, 2, \cdots, n$ by:
\begin{eqnarray*}
M^{u} &:=& \{A = (a_{ij}) \in {\rm M}_n(k) ~|~ a_{ii} = 0, \,\,
{\rm for \,\, all} \,\, i = 1, \cdots, n \,\, {\rm and} \,\,
a_{ur} = 0, \,\, {\rm for \, all}
\,\, u < r\} \\
M^{v,w} &:=& \{A = (a_{ij}) \in {\rm M}_n(k) ~|~ \sum_{i=1}^{n}
a_{ii} = 0,\,\, a_{vv} = 0, \,\,
{\rm and} \,\, a_{ws} = 0, \,\, {\rm for \, all} \,\, w \leq s\} \\
\overline{M}^{v,w} &:=& \{A = (a_{ij}) \in {\rm M}_n(k) ~|~
\sum_{i=1}^{n} a_{ii} = 0,\,\, a_{vv} = 0, \,\, {\rm and} \,\,
a_{ws} = 0, \,\, {\rm for \, all} \,\, w \leq s \neq v\}
\end{eqnarray*}

\begin{example}\exlabel{Anamatrici}
Let $k$ be a field of characteristic zero. Then:
$$
{\mathbb G} {\mathbb H}^{2} \, (\mathcal{T}_{n}(k), \, k) \cong
\bigl( \bigcup_{u \in \{1, 2, \cdots, n\}} M^{u} \, \bigl) \,
\sqcup \, \bigl( \bigcup_{v, w \in \{1, 2, \cdots, n\}, \, v<w}
M^{v,w} \times k^{n-w} \, \bigl) \, \sqcup \, U \sqcup k^{*}
$$
where we denote: $U := \bigcup_{v, w \in \{1, 2, \cdots, n\},\,
v>w} \,\, \overline{M}^{v,w} \times k^{n-w} \times k^{v-w-1}$. The
equivalence classes of $\bigl(\frac{n(n+1)}{2} +
1\bigl)$-dimensional algebras that have an algebra projection on
$\mathcal{T}_{n}(k)$ are the families of algebras having $\{f, \,
e_{ij} ~|~ i, \, j = 1, 2, \cdots, n, \, i \leq j\}$ as a basis
over $k$ and the multiplication given below (we only write down
the non-zero products):
\begin{eqnarray*}
& \mathcal{T}_{n}(k)^{u}_{A} : & \quad e_{it} \star e_{ts} =
e_{il} - \alpha_{is} f, \,\, e_{uj} \star e_{jl} = e_{ul},\,\, e_{uu} \star e_{kl} =
\alpha_{kl} f,\,\, e_{ij} \star e_{uu} = \alpha_{ij} f,\\
&& \quad e_{ij} \star e_{jj} = e_{ij} - \alpha_{ij} f, \,\, e_{iu}
\star e_{uu} = e_{iu}, \,\, e_{uu} \star f = f \star e_{uu} = f,\\
&& \quad {\rm where} \,u \in \{1, 2, \cdots, n\}, \, A = (\alpha_{pq})_{p, q =
\overline{1,n}} \in M^{u}, \, i, j \neq u, \, t\neq s;\\
& \mathcal{T}_{n}(k)^{v,w}_{B,\Gamma}: & \quad e_{ip} \star e_{pl}
= e_{il} - \beta_{il}f, \,\, e_{wi} \star e_{il} = e_{wl} -
\beta_{l} f,\,\, e_{ww}
\star e_{kl} = e_{wl} \delta^{w}_{k} + \beta_{kl} f,\\
&& \quad  e_{it} \star e_{vv} = \beta_{it} f,\,\, e_{ij} \star e_{jj} =
e_{ij} - \beta_{ij}(1-\delta^{v}_{j})f, \,\, e_{ws} \star e_{vv} = - \gamma_{s} f,\\
&& \quad e_{wj} \star e_{jj} = e_{wj} +
\gamma_{j}(1-\delta^{w}_{j})f, \,\, e_{vv} \star f = e_{ww} \star f = f,\\
&& \quad {\rm where} \, v, w \in \{1, 2, \cdots, n\}, \, v<w, \, B
= (\beta_{pq})_{p, q = \overline{1,n}} \in M^{v,w},\\
&& \quad \Gamma = (\gamma_{r})_{w<r},\, i \neq w, \, p \neq l, \,
i \neq l,\, t \neq v, \, s
\notin \{v,w\};\\
\end{eqnarray*}
\begin{eqnarray*}
& \mathcal{T}_{n}(k)^{w,v}_{C,\Psi,\Omega} : & \quad e_{ip} \star
e_{pl} = e_{il} - \delta_{il} f, \, e_{wi} \star e_{ij} = e_{wj} +
\psi_{j} f, \, e_{wt} \star e_{tv} = e_{wv} + \omega_{t} f,\\
&& \quad e_{ww} \star e_{kl} = e_{wl} \delta_{k}^{w} +
\delta_{kl}f, \, e_{ij} \star e_{jj} = e_{ij} - \delta_{ij} f,\,
e_{wj} \star e_{vv} = - \psi_{j} f\\
&& \quad e_{iv} \star e_{vv} = e_{iv}, \, e_{vv} \star f = f \star
e_{ww} = f, \,{\rm where} \,  v, w \in \{1, 2, \cdots, n\},\\
&& \quad v>w, \, C = (\delta_{pq})_{p, q = \overline{1,n}} \in
\overline{M}^{v,w},\, \Psi = (\psi_{r})_{w<r}, \, \Omega =
(\omega_{s})_{w<s<v},\\
&& \quad i \neq w,\, j \neq v,\, p \neq l, \, t \notin \{v, w\};\\
& {}^{\lambda}\mathcal{T}_{n}(k) : & \quad e_{ij} \star e_{jl} =
e_{il} + \lambda^{-1} \delta^{n}_{i} \delta^{n}_{l}
(\delta^{n}_{j} - 1)f, \, e_{nn} \star f = f \star e_{nn} = f, \,
f^{2} = \lambda f,\\
&& \quad {\rm where} \, \lambda \in k^{*}.
\end{eqnarray*}
Indeed, we start by discussing the algebra maps $\lambda:
\mathcal{T}_{n}(k) \to k$. Denote $\lambda(e_{ij}) = \alpha_{ij}
\in k$, for all $e_{ij} \in B$. Since $e_{ii}^{2} = e_{ii}$ we
have $\alpha_{ii} \in \{0,\, 1\}$, for all $i = 1, \cdots, n$.
Moreover, since ${\rm char} (k) = 0$ and we assume $\lambda$ to be
unitary it follows that $\Sigma_{i=1}^{n} \alpha_{ii} = 1$.
Therefore, $\alpha_{uu} = 1$, for some $u \in \{1, \, \cdots, n
\}$ and $\alpha_{ii} = 0$, for all $i \neq u$; we denote by
$\lambda^{u}$ this algebra map. As $e_{ij} = e_{ii} e_{ij}$, for
all $i \leq j$, we obtain that $\alpha_{ij} = 0$, for all $i \neq
u$ and $i \leq j$. Finally, since $e_{uj} e_{uu} = 0$ and
$\lambda^{u}(e_{uu}) = 1$, we obtain that $\alpha_{uj} = 0$, for
any $u < j$. To conclude, the set of algebra maps
$\mathcal{T}_{n}(k) \to k$ are in bijection to the set $\{1,
\cdots, n\}$ and the algebra map corresponding to some $j \in \{1,
\cdots, n\}$ is given by $\lambda^{j} (e_{jj}) := 1$ and
$\lambda^{j} (e_{uv}) := 0$, for all $(u, v) \neq (j, j)$. The
next step of the proof is a computational one: namely, for any
$u$, $v \in \{1, \, 2,\, \cdots,\, n\}$ we are left to compute the
set of all $(\lambda^{v},\, \Lambda^{u})$-cocycles $\vartheta$.
This is achieved by straightforward but lengthy checking of
\equref{ciudat1} which in this case comes down to the following
compatibility condition:
\begin{equation}\eqlabel{suptr}
\vartheta(e_{ij},\, e_{rs}e_{pq}) - \vartheta(e_{ij} e_{rs},\,
e_{pq}) = \vartheta(e_{ij},\, e_{rs}) \Lambda^{u} -
\vartheta(e_{rs},\,e_{pq}) \lambda^{v}
\end{equation}
with $i \leq j$, $r \leq s$ and $p \leq q$. Rather than including
here the cumbersome computations we will just point out the main
steps taken; the detailed proof can be provided upon request.
First, notice that since \equref{suptr} is not symmetric with
respect to the maps $\lambda^{v}$ and $\Lambda^{u}$ we distinguish
three cases, namely: $u=v$, $u<v$ and respectively $u>v$. For the
case $u=v$ the $(\lambda^{u},\, \Lambda^{u})$-cocycles obtained
are implemented by a family of $(n-u)$ scalars and a matrix of
trace zero with zeros on the line $u$ strictly above the diagonal.
Then, it can be proved that any such cocycle is equivalent (in the
sense of \equref{ecfolo}) to a cocycle implemented by a matrix in
$M^{u}$. The corresponding co-flag algebras are those denoted by
$\mathcal{T}_{n}(k)^{u}_{A}$. If $u<v$ then any $(\lambda^{v},\,
\Lambda^{u})$-cocycle is equivalent to a cocycle implemented by
$(n-v)$ scalars and a matrix in $M^{u,v}$. The corresponding
co-flag algebras are those denoted by
$\mathcal{T}_{n}(k)^{u,v}_{B,\Gamma}$. Finally, $u>v$ then any
$(\lambda^{v},\, \Lambda^{u})$-cocycle is equivalent to a cocycle
implemented by two families of $(n - v)$ and respectively $(u - v
- 1)$ scalars and a matrix in $\overline{M}^{u,v}$. The
corresponding co-flag algebras are those denoted by
$\mathcal{T}_{n}(k)^{v,u}_{C,\Psi,\Omega}$. Finally, the last
family of co-flag algebras, denoted by
${}^{\lambda}\mathcal{T}_{n}(k)$, corresponds to a co-flag datum
of the second kind associated $(\delta^{n}, \, \lambda)$, where
$\delta^{n}(i) = \delta^{n}_{i}$ is the Kronecker symbol and
$\lambda \in k^{*}$.

Moreover, for $n = 2$ we can also write down in a transparent way
the other classifying object, namely ${\mathbb H} {\mathbb O}
{\mathbb C} \, (\mathcal{T}_{2}(k), \, k)$. By a long but
straightforward computation it can easily be seen that ${\mathbb
H} {\mathbb O} {\mathbb C} \, (\mathcal{T}_{2}(k), \, k)$ contains
the algebras whose multiplication is depicted below together with
the direct product of algebras $\mathcal{T}_{2}(k) \times k$:
\begin{center}
\begin{tabular} {c | c  c  c  c  }
$\star$ & $e_{11}$ & $e_{12}$ & $e_{22}$ & $f$\\
\hline $e_{11}$ & $e_{11}$ & $e_{12}$ & 0 & $f$\\
$e_{12}$ & 0 & 0 & $e_{12}$ & 0 \\
$e_{22}$ & 0 & 0 & $e_{22}$ & 0\\
$f$ & $f$ & 0 & 0 & 0 \\
\end{tabular} \,\,\,
\begin{tabular} {c | c  c  c  c  }
$\star$ & $e_{11}$ & $e_{12}$ & $e_{22}$ & $f$\\
\hline $e_{11}$ & $e_{11}$ & $e_{12}$ & 0 & 0\\
$e_{12}$ & 0 & 0 & $e_{12}$ & 0 \\
$e_{22}$ & 0 & 0 & $e_{22}$ & $f$\\
$f$ & 0 & 0 & $f$ & 0 \\
\end{tabular}
\,\,\,
\begin{tabular} {c | c  c  c  c  }
$\star$ & $e_{11}$ & $e_{12}$ & $e_{22}$ & $f$\\
\hline $e_{11}$ & $e_{11}-f$ & $e_{12}$ & $f$ & 0\\
$e_{12}$ & 0 & 0 & $e_{12}$ & 0 \\
$e_{22}$ & $-f$ & 0 & $e_{22}-f$ & $f$\\
$f$ & 0 & 0 & $f$ & 0 \\
\end{tabular}
\end{center}

\begin{center}
\begin{tabular} {c | c  c  c  c  }
$\star$ & $e_{11}$ & $e_{12}$ & $e_{22}$ & $f$\\
\hline $e_{11}$ & $e_{11}$ & $e_{12}$ & 0 & $f$\\
$e_{12}$ & 0 & 0 & $e_{12}$ & 0 \\
$e_{22}$ & 0 & 0 & $e_{22}$ & 0\\
$f$ & 0 & 0 & $f$ & 0 \\
\end{tabular} \,\,\,
\begin{tabular} {c | c  c  c  c  }
$\star$ & $e_{11}$ & $e_{12}$ & $e_{22}$ & $f$\\
\hline $e_{11}$ & $e_{11}$ & $e_{12}-f$ & 0 & $f$\\
$e_{12}$ & $f$ & 0 & $e_{12}-f$ & 0 \\
$e_{22}$ & 0 & $f$ & $e_{22}$ & 0\\
$f$ & 0 & 0 & $f$ & 0 \\
\end{tabular}
\,\,\,
\begin{tabular} {c | c  c  c  c  }
$\star$ & $e_{11}$ & $e_{12}$ & $e_{22}$ & $f$\\
\hline $e_{11}$ & $e_{11}$ & $e_{12}$ & 0 & 0\\
$e_{12}$ & 0 & 0 & $e_{12}$ & 0 \\
$e_{22}$ & 0 & 0 & $e_{22}$ & $f$\\
$f$ & $f$ & 0 & 0 & 0 \\
\end{tabular}
\end{center}
\begin{center}
\begin{tabular} {c | c  c  c  c  }
$\star$ & $e_{11}$ & $e_{12}$ & $e_{22}$ & $f$\\
\hline $e_{11}$ & $e_{11}$ & $e_{12}+f$ & 0 & 0\\
$e_{12}$ & 0 & 0 & $e_{12}$ & 0 \\
$e_{22}$ & 0 & 0 & $e_{22}$ & $f$\\
$f$ & $f$ & 0 & 0 & 0 \\
\end{tabular}
\end{center}
\end{example}

\section{Applications to coalgebras and Poisson algebras} \selabel{aplicatii}
In this section we shall present two applications of our results
to the theory of coalgebras and respectively Poisson algebras, the
latter being the algebraic counterpart of Poisson manifolds. Our
strategy is to use two different contravariant functors which have
both the category of algebras as a codomain, namely the linear
dual functor $(-)^* := {\rm Hom}_k \, (-, \, k)$ and respectively
${\rm Fun} \, (-) := C^{\infty} (-)$ the real smooth functions on
a Poisson manifold functor.

\subsection*{Supersolvable coalgebras}
We recall that a coalgebra $C = (C, \, \Delta, \, \varepsilon)$ is
a vector space $C$ equipped with a comultiplication $\Delta : C
\to C \otimes C$ and a counit $\epsilon : C \to k$ such that
$(\Delta \ot {\rm Id}) \circ \Delta = ({\rm Id} \ot \Delta) \circ
\Delta$ and $({\rm Id} \ot \varepsilon) \circ \Delta =
(\varepsilon \ot {\rm Id}) \circ \Delta = {\rm Id}$, where $\ot =
\ot_k$ and ${\rm Id}$ is the identity map on $C$. We use the
$\Sigma$-notation for comultiplication: $\Delta (c) = c_{(1)} \ot
c_{(2)}$, for all $c\in C$ (summation understood). The base field
$k$, with the obvious structures, is the final object in the
category of coalgebras. We only provide some basic information of
coalgebras, referring the reader to \cite{BW} for more detail. If
$C = (C, \, \Delta, \, \varepsilon)$ is a colgebra, then the
linear dual $C^* = {\rm Hom}_k (C, \, k)$ is an associative
algebra under the convolution product $(f\ast g) (c) : = f
(c_{(1)}) g(c_{(2)})$, for all $f$, $g\in C^*$ and $c\in C$ having
the unit $1_{C^*} = \varepsilon$. Conversely, if $A$ is a finite
dimensional algebra with a basis $\{e_i \, | \, i = 1, \cdots, n
\}$, then the linear dual $A^*$ is a coalgebra with the
comultiplication and the counit given for any $i = 1, \cdots, n$
by:
\begin{equation}\eqlabel{coalg}
\Delta_{A^*} (e_i^*) := \sum_{j, l = 1}^n \, e_i^* (e_j e_l) \,
e_j^* \ot e_l^*, \qquad \varepsilon_{A^*} (e_i^*) := e_i^*(1_A)
\end{equation}
where $\{e_i^* \, | \, i = 1, \cdots, n \}$ is the dual basis of
$\{e_i \, | \, i = 1, \cdots, n \}$. The contravariant functor
$(-)^* := {\rm Hom}_k \, (-, \, k)$ gives a duality between the
category of all finite dimensional coalgebras and the category of
finite dimensional algebras \cite{BW}. Having the supersolvable
Lie algebras \cite{barnes} as a source of inspiration we introduce
the following concept:

\begin{definition} \delabel{superre}
A coalgebra $C$ is called \emph{supersolvable} if there exists a
positive integer $n$ such that $C$ has a finite chain of
subcoalgebras
\begin{equation} \eqlabel{sir}
k \cong C_1 \subset C_2 \subset \cdots \subset C_n := C
\end{equation}
such that $C_i$ has codimension $1$ in $C_{i+1}$, for all $i = 1,
\cdots, n-1$.
\end{definition}

Any supersolvable coalgebra is finite dimensional and the
definition is equivalent to the fact that ${\rm dim} (C_i) = i$,
for all $i = 1, \cdots, n$. Furthermore, since any supersolvable
coalgebra $C$ contains a $1$-dimensional subcolagebra, we obtain
that $G(C) \neq \emptyset$, where $G(C)$ is the space of
group-like elements of $C$. Using the duality given by the functor
$(-)^*$ it follows that a finite dimensional coalgebra $C$ is
supersolvable if and only if its convolution algebra $C^*$ is a
co-flag algebra in the sense of \deref{coflaglbz} and vice-versa:
a finite dimensional algebra $A$ is a co-flag algebra if and only
if its dual $A^*$ is a supersolvable coalgebra. Thus, the results
obtained in the previous section can be applied for the
classification problem of supersolvable coalgebras of a given
dimension.\footnote{The classification of solvable Lie algebras
\cite{gra2}, over arbitrary fields, was achieved up to dimension
$4$.} Using the above facts, \coref{clasfinal} classifies in fact
all $3$-dimensional supersolvable coalgebras: the isomorphism
classes are the duals $A^*$ of the algebras listed in the
statement. For example, the dual coalgebra associated to the
noncommutative algebra $k < x, \, y \, | \, x^2 = 1, \, y^2 = 0,
\, xy = -yx = y
>$ is the non-cocommutative supersolvable coalgebra having $\{f_1,
\, f_2, \, f_3\}$ as a basis and the comultiplication and the
counit given by:
\begin{eqnarray*}
&& \Delta (f_1) := f_1 \ot f_1 + f_2 \ot f_2, \quad \Delta (f_2)
:= f_1 \ot f_2 + f_2\ot f_1, \quad \, \varepsilon (f_1) := 1, \quad \varepsilon(f_2) := 0 \\
&& \Delta (f_3) := f_1 \ot f_3 + f_3 \ot f_1 + f_2 \ot f_3 - f_3
\ot f_2, \qquad \varepsilon (f_3) := 0
\end{eqnarray*}
The coalgebra obtained in this way is indeed supersolvable by
choosing $C_1 := k (f_1 - f_2)$ (we observe that $f_1 \pm f_2$ is
a group-like element) and $C_2 := k f_1 + k f_2$ as the
intermediary coalgebras of dimension $1$ respectively $2$ in the
sequence \equref{sir}.

Another application of our theory of \seref{glalg} is the
following: let $C$ be a finite dimensional coalgebra, $n$ a
positive integer and $V := k^n$. Then by taking the convolution
algebra $A : = C^*$ we obtain that the object ${\mathbb G}
{\mathbb H}^{2} \, (C^*, \, k^n)$ classifies, up to an isomorphism
which stabilizes $C$ and costabilizes $k^n$, all coalgebras which
contain $C$ as a subcoalgebra of codimension $n$. Moreover, a
coalgebra $D$ contains $C$ as a subcoalgebra of codimension $n$ if
and only if $D$ is isomorphic to the dual coalgebra $(C^* \star \,
k^n)^*$ of a Hochschild product $C^* \star \, k^n$ between the
convolution algebra $C^*$ and the vector space $k^n$. The formula
for the comultiplication of any such coalgebra $(C^* \star \,
k^n)^*$ can be written down effectively by using \equref{coalg}
and \equref{hoproduct2}. This observation shows that the
GE-problem applied for finite dimensional algebras and finite
dimensional vector spaces gives the answer at the level of finite
dimensional coalgebras to what we have called the \emph{extending
structures problem} studied in \cite{am-2015b, am-2013, am-2014}
for Jacobi, Lie and respectively associative algebras.

\subsection*{Applications to Poisson algebras}
Commutative Poisson algebras are algebraic counterparts of Poisson
manifolds from differential geometry: for a given smooth manifold
$M$, there is a one-to-one correspondence between Poisson brackets
on the commutative algebra $P := C^{\infty} (M)$ of smooth
functions on $M$ and all Poisson structures on $M$ \cite{gra2013}.
The importance of Poisson algebras in several areas of mathematics
and physics (Hamiltonian mechanics, differential geometry, Lie
groups, noncommutative algebraic/diferential geometry,
(super)integrable systems, quantum field theory, vertex operator
algebras, quantum groups) it is also well known -- see
\cite{gra1993, gra2013, LPV} and the references therein. In fact,
$ C^{\infty} (-)$ gives a contravariant functor from the category
of Poisson manifolds to the category of Poisson algebras and
constitutes the tool through which geometrical problems can be
restated and approached at the level of Poisson algebras. In
particular, the GE-problem, formulated for Poisson algebras, is
just the algebraic counterpart of the following question from
differential geometry: \emph{Let $M$ be a Poisson manifold.
Describe and classify all Poisson manifolds containing $M$ as a
manifold of a given codimension.}

We recall that a \emph{Poisson algebra} is a triple $P = (P, \,
m_P, \, [-, \, -])$, where $(P, m_P)$ is a unital associative
algebra, $(P, \, [-, \, -])$ is a Lie algebra such that the
Leibniz law holds for any $p$, $q$, $r\in P$:
\begin{equation}\eqlabel{p1}
[pq, \, r ] = [p, \, r] \, q + p \, [q, \, r]
\end{equation}
Usually, a Poisson algebra $P$ is by definition assumed to be
commutative like the algebra $P := C^{\infty} (M)$ of real smooth
functions on a Poisson manifold $M$ which is the typical example
of a Poisson algebra. However, following \cite{Farkas, kobo} in
order to broaden the class of Poisson algebras and to be able to
construct relevant examples, throughout this paper we do not
impose this restriction. A morphism between two Poisson algebras
$P$ and $P'$ is a linear map $\varphi: P \to P'$ that is a
morphism of associative algebras as well as of Lie algebras; we
will denote by ${\rm Aut}_{\rm Poss} (P)$ the group of
automorphisms of a Poisson algebra $P$. For basic concepts and
unexplained notions on Lie algebras we refer to \cite{EW} and to
\cite{LPV} for those concerning Poisson algebras.

Let $P = (P, \, m_P)$ be an algebra and $u\in k$. Then $(P, m_P,
[-, \, -]_u)$ is a Poisson algebra, where $[a, \, b]_u := u (ab -
ba)$, for all $a$, $b\in P$. In particular, any associative
algebra $P = (P, \, m_P)$ is a Poisson algebra with the abelian
Lie bracket, i.e. $[a, \, b] := 0$, for all $a$, $b\in P$. On the
other hand, let $\mathfrak{g} = (\mathfrak{g}, \, [-, \,
-]_{\mathfrak{g}} )$ be a Lie algebra with a linear basis $\{e_i
\, | \, i\in I \}$. Then the symmetric algebra $P := S
(\mathfrak{g})$ of $\mathfrak{g}$ (i.e. the polynomial algebra
$k[e_i \, | \, i\in I ]$) is a Poisson algebra with the bracket
defined by $[e_i, \, e_j] := [e_i, \, e_j]_{\mathfrak{g}}$, for
all $i$, $j \in I$ and extended to the entire algebra $k[e_i \, |
\, i\in I ]$ via the Leibniz law \equref{p1}.

If $P$ is a Poisson algebra then the direct product $P\times k$ is
a Poisson algebra with the direct product structures of
associative/Lie algebra: that is the multiplication and the
bracket is defined for any $p$, $q\in P$ and $x$, $y\in k$ by:
\begin{equation} \eqlabel{directproduct}
(p, x) \star (q, y) := (p q, \,\, xy), \qquad \{ (p, x), \, (q, y)
\} := \bigl( [p, q], \,\, 0 \bigl)
\end{equation}
Furthermore, the canonical projection $\pi_P : P \times k \to P$,
$\pi_P (p, \, x) := p$ is a surjective Poisson algebra map having
a $1$-dimensional kernel. It what follows we shall classify all
Poisson algebras $Q$ that admit a surjective Poisson algebra map
$Q \to P \to 0$ with a $1$-dimensional kernel. In order to do this
we first recall from \cite{am-2015} the following concept:

\begin{definition} \delabel{coflagab}
A \emph{co-flag datum} of a Poisson algebra $P$ is a $5$-tuple
$(\lambda, \, \Lambda, \, \vartheta, \, \gamma, \, f)$, where
$\lambda$, $\Lambda$, $\gamma: P \to k$ are linear maps,
$\vartheta$, $f : P\times P \to k$ are bilinear maps such that:
\begin{enumerate}
\item[(CF1)] $(\lambda, \, \Lambda, \, \vartheta)$ is a co-flag
datum of the first kind of the associative algebra $P$

\item[(CF2)] $\lambda( [p, \, q]) = \Lambda ([p, \, q]) =  \gamma
([p, \, q]) = f (p, p) = 0$

\item[(CF3)] $f(p, [q, \, r ] ) + f(q, [r, \, p ]) + f(r, [p, \,
q] ) + \gamma (p) f(q, r) + \gamma (q) f(r, p) + \gamma (r) f(p,
q) = 0$

\item[(CF4)] $f(pq, \, r) - \Lambda(q) f(p, \, r) - \lambda(p)
f(q, \,r) = \gamma (r) \vartheta (p, \, q) + \vartheta ([p, \, r],
\, q) + \vartheta (p, \, [q, \, r])$

\item[(CF5)] $\gamma (pq) = \gamma (p) \Lambda(q) + \lambda(p)
\gamma(q) $
\end{enumerate}
for all $p$, $q$, $r\in P$. We denote by ${\mathcal F} \, (P)$ the
set of all co-flag data of $P$.
\end{definition}
The above concept was introduced in \cite[Definition 3.2]{am-2015}
under the name of 'abelian co-flag datum' of $P$. We dropped the
adjective 'abelian' since the other family of co-flag datum
introduced in \cite[Definition 3.4]{am-2015} will generate a
family of Poisson algebras which are all isomorphic to the direct
product $P\times k$ (see the proof of \thref{pcodim1} below) --
hence, they are irrelevant from the classification view point.

Let $(\lambda, \, \Lambda, \, \vartheta, \, \gamma, \, f) \in
{\mathcal F} \, (P)$ be a co-flag datum of a Poisson algebra $P$.
Then we shall denote by $P_{(\lambda, \, \Lambda, \, \vartheta, \,
\gamma, \, f)} := P \times k$ the direct product of vector spaces
which is a Poisson algebra with the multiplication $\star$ and the
bracket $\{-, \, -\}$ defined \cite[Section 3]{am-2015} for any
$p$, $q\in P$, $x$, $y \in k$ by:
\begin{eqnarray}
(p, x) \star (q, y) &:=& (p q, \,\, \vartheta (p, q) + \lambda (p)
y + \Lambda (q) x )
\eqlabel{Icrpos1} \\
\{ (p, x), \, (q, y) \} &:=& \bigl( [p, q], \,\, f(p, q) + \gamma
(p) y - \gamma (q) x \bigl) \eqlabel{Icrpos2}
\end{eqnarray}

Now we can prove the following classification result: it is the
Poisson version of \thref{clasres} and improves the classification
given in \cite[Theorem 3.6]{am-2015} where all Poisson algebras
$Q$ having a Poisson surjection $Q \to P \to 0$ with a
$1$-dimensional kernel are classified in a more restrictive
fashion: up to an isomorphism which stabilizes $k$ and
co-stabilizes $P$.

\begin{theorem} \thlabel{pcodim1}
Let $P$ be a Poisson algebra. Then:

$(1)$ A Poisson algebra $Q$ has a surjective Poisson algebra map
$Q \to P \to 0$ with a $1$-dimensional kernel if and only if $Q
\cong P \times k$ or $Q \cong P_{(\lambda, \, \Lambda, \,
\vartheta, \, \gamma, \, f)}$, for some $(\lambda, \, \Lambda, \,
\vartheta, \, \gamma, \, f) \in {\mathcal F} \, (P)$.

$(2)$ Two Poisson algebras $P_{(\lambda, \, \Lambda, \, \vartheta,
\, \gamma, \, f)}$ and $P_{(\lambda', \, \Lambda', \, \vartheta',
\, \gamma', \, f')}$ are isomorphic if and only if there exists a
triple $(s_0, \, \psi, \, r) \in k^* \times {\rm Aut}_{\rm Poss}
(P) \times {\rm Hom}_k (P, \, k)$ such that for any $p$, $q \in
P$:
\begin{eqnarray}
&& \lambda = \lambda' \circ \psi, \qquad \Lambda = \Lambda' \circ \psi,
\qquad \gamma = \gamma' \circ \psi \eqlabel{P2015a} \\
&& \vartheta (p, q) \, s_0 = \vartheta' \bigl(\psi(p), \psi(q)
\bigl) + \lambda(p) r(q) + \Lambda(q) r(p) - r(pq)
\eqlabel{P2015b} \\
&& f(p, q) \, s_0 = f' \bigl(\psi (p), \, \psi(q)\bigl) + \gamma
(p) r(q) - \gamma (q) r(p) - r([p, \, q]) \eqlabel{P2015c}
\end{eqnarray}

$(3)$ The Poisson algebras $P_{(\lambda, \, \Lambda, \, \vartheta,
\, \gamma, \, f)}$ and $P \times k$ are not isomorphic.
\end{theorem}

\begin{proof}
$(1)$ Let $Q$ be a Poisson algebra having a surjective Poisson
algebra map $Q \to P \to 0$ with $1$-dimensional kernel. Then,
using \cite[Proposition 2.4 and Proposition 3.5]{am-2015} we
obtain that $Q \cong P_{(\lambda, \, \Lambda, \, \vartheta, \,
\gamma, \, f)}$, for some $(\lambda, \, \Lambda, \, \vartheta, \,
\gamma, \, f) \in {\mathcal F} \, (P)$ or $Q \cong P^{(\lambda, \,
\vartheta, \, u)}$, where $u \in k \setminus \{0\}$, $\lambda: P
\to k$ is a unit preserving linear map, $\vartheta : P\times P \to
k$ is a bilinear map satisfying the following two compatibilities
for any $p$, $q$, $r\in P$:
\begin{equation} \eqlabel{noncofl}
\lambda (pq) = \lambda(p) \lambda(q) - u \, \theta (p, \, q),
\quad \theta (p, \, qr) - \theta (pq, \, r) = \theta (p, \, q)
\lambda(r) - \theta (q, \, r) \lambda(p)
\end{equation}
which we called in \cite[Definition 3.4]{am-2015} a non-abelian
co-flag datum of $P$. The Poisson algebra $P^{(\lambda, \,
\vartheta, \, u)}$ is the vector space $P \times k$ having the
multiplication and the Poisson bracket defined by:
\begin{eqnarray}
(p, x) \star (q, y) &:=& (p q, \,\, \vartheta (p, q) + \lambda (p)
y + \lambda (q) x + u\, xy)
\eqlabel{IIcrpos1} \\
\{ (p, x), \, (q, y) \} &:=& \bigl( [p, q], \,\, -u^{-1} \,
\lambda ([p, \, q]) \bigl) \eqlabel{IIcrpos2}
\end{eqnarray}
for all $p$, $q\in P$ and $x$, $y\in k$. Since $u \neq 0$ we
obtain from the first equation of \equref{noncofl} that
$\vartheta$ is implemented by $u$ and $\lambda$ by the formula
$\vartheta (p, \, q) = u^{-1} \Bigl( \lambda(p) \lambda(q) -
\lambda(pq) \Bigl)$ and hence the multiplication of the Poisson
algebra $P^{(\lambda, \, u)}$ given by \equref{IIcrpos1} takes the
form:
$$
(p, x) \star (q, y) = (p q, \,\, u^{-1} \Bigl( \lambda(p)
\lambda(q) - \lambda(pq) \Bigl) + \lambda (p) y + \lambda (q) x +
u\, xy)
$$
which is precisely \equref{patratarbitar}. The first part is
finished once we observe that the map given by the formula
\equref{izocaut}, namely $\varphi : P^{(\lambda, \, u)} \to P
\times k$, $\varphi (p, x) := (p, \, \lambda (p) + u\, x)$, for
all $p\in P$ and $x\in k$ is an isomorphism of Poisson algebras.
Indeed, \reref{nucleunul} shows that $\varphi$ is an isomorphism
of associative algebras; hence we only have to prove that it is
also a Lie algebra map, where the bracket on $P^{(\lambda, \, u)}$
is given by \equref{IIcrpos2}. A straightforward computation shows
that $ \varphi \Bigl( \{ (p, \, x), \, (q, \, y) \}\Bigl) =
[\varphi(p, \, x), \, \varphi(q, \, y) ]_{P\times k} = \bigl([p,
\, q], \, 0 \bigl)$ and thus any Poisson algebra $P^{(\lambda, \,
u)}$ is in fact isomorphic to the direct product of Poisson
algebras $P\times k$.

$(2)$ The first step in proving \thref{clasres} gives a bijection
between all associative algebra isomorphism corresponding to the
Poisson algebras $P_{(\lambda, \, \Lambda, \, \vartheta, \,
\gamma, \, f)}$ and $P_{(\lambda', \, \Lambda', \, \vartheta', \,
\gamma', \, f')}$ and the set of all triples $(s_0, \, \psi, \, r)
\in k^* \times {\rm Aut}_{\rm Alg} (P) \times {\rm Hom}_k (P, \,
k)$ satisfying \equref{P2015b} and the first two compatibilities
of \equref{P2015a}. The bijection is given such that the
associative algebra isomorphism $\varphi = \varphi_{(s_0, \, \psi,
\, r)}: P_{(\lambda, \, \Lambda, \, \vartheta, \, \gamma, \, f)}
\to P_{(\lambda', \, \Lambda', \, \vartheta', \, \gamma', \, f')}$
associated to $(s_0, \, \psi, \, r)$ is given by the formula
\equref{4morfb}, that is $\varphi (p, \, x) = \bigl( \psi(p), \,
r(p) + x\, s_0 \bigl)$, for all $p\in P$ and $x\in k$. The proof
will be finished if we show that such a map $\varphi =
\varphi_{(s_0, \, \psi, \, r)}$ is also a morphism of Lie algebras
if and only if $\psi$ is an automorphism of the Lie algebra $P =
(P, \, [-, \, -])$ and the last equation of \equref{P2015a} and
\equref{P2015c} hold. This is an elementary fact: by a
straightforward computation we can show that $\varphi \Bigl( \{(p,
0), \, (q, 0) \} \Bigl) = \{ \varphi (p, 0), \, \varphi (q, 0) \}$
if and only if $\psi: P \to P$ is a Lie algebra map and
\equref{P2015c} holds. In a similar fashion $\varphi \Bigl( \{(p,
0), \, (0, x) \} \Bigl) = \{ \varphi (p, 0), \, \varphi (0, x) \}$
if and only if $\gamma = \gamma' \circ \psi$. The rest of the
details are left to the reader.

$(3)$ Follows from step (2) of the proof of \thref{clasres} which
proves that the associative algebras $P_{(\lambda, \, \Lambda, \,
\vartheta, \, \gamma, \, f)}$ and $P \times k$ are never
isomorphic.
\end{proof}

\begin{corollary} \colabel{perfect}
Let $P$ be a Poisson algebra for which there is no algebra map $P
\to k$ or $P$ is perfect as a Lie algebra, i.e. $P = [P, \, P]$.
Then, up to an isomorphism, the only Poisson algebra $Q$ for which
there exists a surjective algebra map $Q \to P \to 0$ having a
$1$-dimensional kernel is the direct product $P \times k$ of
Poisson algebras.
\end{corollary}

\begin{proof}
The proof follows from \thref{pcodim1} since in both cases the set
${\mathcal F} \, (P)$ is empty. Indeed, the first case follows
form \coref{casspecial}, while if $P$ is perfect as a Lie algebra
then, using the compatibilities (CF2) we obtain that $\lambda =
\Lambda \equiv 0$, the trivial maps, which contradicts the fact
that the algebra maps $\lambda$ and $\Lambda$ are unit preserving.
\end{proof}

The geometrical meaning of \coref{perfect} is the following: if
$M$ is a real Poisson manifold such that the algebra $C^{\infty}
(M)$ of all real smooth functions on $M$ is perfect as a Lie
algebra, then up to an isomorphism, there is only one Poisson
manifold containing $P$ as a sub-manifold of codimension $1$. The
group of Poisson algebra automorphisms of $P_{(\lambda, \,
\Lambda, \, \vartheta, \, \gamma, \, f)}$ can also be described.
Using the proof of \thref{pcodim1}, the Poisson version of
\coref{izoaut} takes the following form:

\begin{corollary} \colabel{Pizoaut}
Let $P$ be a Poisson algebra, $(\lambda, \, \Lambda, \, \vartheta,
\, \gamma, \, f) \in {\mathcal F} \, (P)$ a co-flag datum of $P$
and let ${\mathcal G} {\mathcal P} \bigl(P, \, (\lambda, \,
\Lambda, \, \vartheta, \, \gamma, \, f) \bigl)$ be the set of all
triples $(s_0, \, \psi, \, r) \in k^* \times {\rm Aut}_{\rm Poss}
(P) \times P^*$ satisfying the compatibility conditions
\equref{P2015a}-\equref{P2015c} written for $\lambda' = \lambda$,
$\Lambda' = \Lambda$, $\gamma' = \gamma$, $\vartheta' = \vartheta$
and $f' = f$. Then, there exists an isomorphism of groups
$$
{\rm Aut}_{\rm Poss} \bigl(P_{(\lambda, \, \Lambda, \, \vartheta,
\, \gamma, \, f)} \bigl) \, \cong \, {\mathcal G} {\mathcal P} \,
\bigl(P, \, (\lambda, \, \Lambda, \, \vartheta, \, \gamma, \, f)
\bigl)
$$
where the latter is a group with respect to the following
multiplication:
$$
(s_0, \, \psi, \, r) \cdot (s_0', \, \psi', \, r') := (s_0 s_0',
\, \psi\circ \psi', \, r \circ \psi' + s_0 r' )
$$
for all $(s_0, \, \psi, \, r)$, $(s_0', \, \psi', \, r') \in \in
{\mathcal G} {\mathcal P} \, \bigl(P, \, (\lambda, \, \Lambda, \,
\vartheta, \, \gamma, \, f) \bigl)$. Moreover, the map
$$
{\mathcal G} {\mathcal P} \, \bigl(P, \, (\lambda, \, \Lambda, \,
\vartheta, \, \gamma, \, f) \bigl) \longrightarrow P^* \,
\ltimes_{\zeta} \bigl(k^* \times {\rm Aut}_{\rm Poss} (P) \bigl),
\qquad (s_0, \, \psi, \, r) \mapsto \bigl (s_0^{-1} r, \, (s_0, \,
\psi) \bigl)
$$
is an injective morphism of groups.
\end{corollary}

We end the paper with a relevant example which follows by a long
computation based on \thref{pcodim1}; the detailed proof can be
provided upon request.

\begin{example} \exlabel{possheis}
Let $k$ be a field of characteristic zero and $P := {\mathcal H}
(3, \, k)$ the $3$-dimensional non-commutative Heisenberg-Poisson
algebra: i.e., ${\mathcal H} (3, \, k)$ is the set of all upper
triangular $2\times 2$ matrices with the usual multiplication and
the Lie bracket given by $[x, y] := xy - yx$. Consider
$\{e_{11},\, e_{12},\, e_{22}\}$ a basis of ${\mathcal H} (3, \,
k)$ over $k$. Then the set of isomorphism types of all
$4$-dimensional Poisson algebras which admit a surjective Poisson
algebra map on ${\mathcal H} (3, \, k)$ are the ones listed below
together with the usual direct product ${\mathcal H} (3, \, k)
\times k$ (we only write down the non-zero products):
\begin{eqnarray*}
&P_{1}: & \quad e_{11} \star e_{11} =  e_{11}, \,\, e_{11} \star
e_{12} = e_{12}, \,\, e_{12} \star e_{22} = e_{12}, \,\, e_{22}
\star e_{22} = e_{22},\\
&& \quad e_{11} \star f = f \star e_{11} = f, \,\,\,\,  \{e_{11},
\, e_{12}\} = e_{12}, \,\,  \{e_{12}, \, e_{22}\} = e_{12}.\\
&P_{2}: & \quad e_{11} \star e_{11} =  e_{11}, \,\, e_{11} \star
e_{12} = e_{12}, \,\, e_{12} \star e_{22} = e_{12}, \,\,  e_{22}
\star e_{22} = e_{22},\\
&& \quad e_{22} \star f = f \star e_{22} = f, \,\,\, \{e_{11}, \,
e_{12}\} = e_{12}, \,\,  \{e_{12}, \, e_{22}\} = e_{12}\\
&P_{3}: & \quad e_{11} \star e_{11} =  e_{11}, \, \, e_{11} \star
e_{12} = e_{12}-f, \,\, e_{12} \star e_{11} = f, \,\,
e_{12} \star e_{22} = e_{12}-f \\
&& \quad e_{22} \star e_{12} = f, \, \, e_{22} \star e_{22} =
e_{22},\,\, e_{11} \star f = f \star e_{22} = f, \,\, \{e_{11}, \,
e_{12}\} = e_{12},\\
&& \quad \{e_{12}, \,\, e_{22}\} = e_{12}, \,\, \{e_{11}, \, f\} =
f, \,\, \{e_{22},\, f\} = -f.
\end{eqnarray*}
\begin{eqnarray*}
&P_{4}^{\omega}: & \quad e_{11} \star e_{11} =  e_{11}, \, \,
e_{11} \star e_{12} = e_{12}, \,\, e_{12} \star e_{22} = e_{12},
\,\,  e_{22} \star e_{22} = e_{22},\\
&& \quad e_{11} \star f = f \star e_{22} = f, \,\,\,  \{e_{11}, \,
e_{12}\} = e_{12},\, \{e_{12}, \,\, e_{22}\} = e_{12},\,\\
&& \quad \{e_{11},\, f\} = \omega f, \, \, \{e_{22},\, f\} =
-\omega f, \,\, {\rm where} \,\, \omega \in k.\\
&P_{5}^{\tau}: & \quad e_{11} \star e_{11} =  e_{11}, \, \,\,
e_{11} \star e_{12} = e_{12}, \,\, e_{12} \star e_{22} = e_{12},
\,\, e_{22}\star e_{22} = e_{22},\\
&& \quad e_{22} \star f = f \star e_{11} = f, \,\,\,
\{e_{11},\, e_{12}\} = e_{12}, \,\, \{e_{12}, \, e_{22}\} = e_{12},\\
&& \quad \{e_{11}, \, f\} = \tau f,\,\, \{e_{22}, \, f\} = - \tau
f, \,\, {\rm where} \,\, \tau \in k.\\
\end{eqnarray*}
We point out that, even if up to an isomorphism there are only $8$
associative algebras of dimension $4$ with a surjective algebra
map on the algebra ${\mathcal H} (3, \, k)$ (indicated at the end
of \exref{Anamatrici}), the set of isomorphism types of Poisson
algebras having the same algebra structure can be infinite due to
the $1$-parameter families $P_{4}^{\omega}$ and $P_{5}^{\tau}$.
\end{example}

\section*{Acknowledgement}
Parts of this work were undertaken while the first named author
was visiting the Max Planck Institute for Mathematics in Bonn.
Their hospitality and financial support is gratefully
acknowledged.

\end{document}